\documentclass[twoside,phd,etd,12pt]{PSUmath}
\date{December} \Year{2009}  
\author{Melchior Gr{\"u}tzmann}
\title{Courant Algebroids: Cohomology and Matched Pairs}

\usepackage{ifpdf,color}
\usepackage[dvips]{graphicx}
\usepackage[margin=0.3in,labelfont=bf,labelsep=none]{caption}
 \DeclareCaptionFormat{suggested}{\singlespace#1#2 #3\par\doublespace}
\usepackage{fancyhdr}
\usepackage{cite}

\usepackage{url}
\usepackage{makeidx}
\makeindex

\usepackage{amsmath,amsfonts,amssymb,amsthm,amscd}
\newtheorem{vdef}{Definition}[chapter]
\newtheorem{lemma}[vdef]{Lemma}
\newtheorem{prop}[vdef]{Proposition}
\newtheorem{cor}[vdef]{Corollary}
\newtheorem{thm}[vdef]{Theorem}

{\theoremstyle{remark}
\newtheorem{rem}[vdef]{Remark} 
\newtheorem{Exam}[vdef]{Example}  
}
 
\newenvironment{steps}{\begingroup\begin{enumerate}\newcommand\step[1][]%
  {\item[\stepcounter{enumi}\textbf{step~\arabic{enumi}}:]\textbf{##1}}%
}{\end{enumerate}\endgroup}

\newcommand\defbb[2]{\def#1{{\mathbb{#2}}}}
\newcommand\defbf[2]{\def#1{{\mathbf{#2}}}}
\newcommand\defcal[2]{\def#1{{\mathcal{#2}}}}  
\newcommand\deffrak[2]{\def#1{{\mathfrak{#2}}}}
\newcommand\defrm[2]{\def#1{{\mathrm{#2}}}}
\def\defcat{\deffrak}

\def\<{\langle}
\def\>{\rangle}
\def\[{\begin{equation}}  
\def\]{\end{equation}}
\DeclareMathOperator\Ad{Ad}
\DeclareMathOperator\Ann{Ann}
\defcal\B{B}
\defcal\cC{C}  
\defbb\C{C}
\deffrak\c{c}
\DeclareMathOperator\CDO{CDO}
\def\:{\colon}
\def\conn_#1{\nabla_{\! #1}\,}  
\def\ConnectionDer_#1{\nabla_{\!\!#1}\,}
\def\lconn_#1{\nablaleft_{\!\!#1}\,}
\def\rconn_#1{\nablaright_{\!\!#1}\,}
\defrm\ud{d}
\defcal\D{D}
\deffrak\DD{D}
\DeclareMathOperator\dec{dec}
\def\dia{\diamond}
\newcommand\dor[2]{#1\dia #2}
\defcal\E{E}
\DeclareMathOperator\End{End}
\let\oldepsilon=\epsilon
\let\epsilon=\varepsilon
\defcal\F{F}
\deffrak\g{g}
\defcat\Gmfd{Gmfd}
\DeclareMathOperator\braph{graph}
\deffrak\h{h}
\defcal\H{H}
\def\Hald{{\ensuremath{H_{algd}}}}

\def\i{i}  
\DeclareMathOperator\id{id}
\DeclareMathOperator\im{im}
\def\img{\im}  

\def\iso{\cong}
\deffrak\k{k}
\defcal\L{L}
\let\Lie=\L
\defcal\M{M}
\defcal\N{N}
\def\nablaDer{\nabla}
\def\nablaleft{{\overleftarrow{\nabla}}}
\def\nablaright{{\overrightarrow{\nabla}}}
\defcat\Nmfd{Nmfd}
\defbb\NN{N}
\def\oo{\Omega}
\def\w{\mho}
\defcal\P{P}
\def\pfrac#1#2{\frac{\partial#1}{\partial#2}}
\DeclareMathOperator\pr{pr}
\defrm\pt{pt}
\defbb\R{R}
\defbf\br{r}
\def\Rleft{{\overleftarrow{R}}}
\def\Rright{{\overrightarrow{R}}}
\def\rep{\nabla^o}
\DeclareMathOperator\rk{rk}
\defbb\S{S}
\defbf\bs{s}
\DeclareMathOperator\sgn{sgn}
\deffrak\sl{sl}
\defcat\Smfd{Smfd}
\def\smooth{\ensuremath{\cC^\infty}}
\def\ssmooth{{\mathcal{O}}}
\def\gsmooth{\ssmooth}
\def\holom{{\mathcal{O}}}
\let\varsubset=\subset
\let\subset=\subseteq

\let\supset=\supseteq
\defbb\T{T}
\defcal\U{U}
\defcat{\Vect}{Vect}

\defcal\X{X} 
\deffrak\Xx{X}  
\def\vect{\Xx}
\defbb\Z{Z}

\newcommand{\ip}[2]{\langle #1 , #2 \rangle} 
\def\ipG#1#2{\ip{#1}{#2}_\GG}
\def\lbG#1#2{[#1,#2]_\GG}
\def\duality#1#2{\langle #1,#2\rangle}
\newcommand{\beq}[1]{\begin{equation}\label{#1}}
\newcommand{\eeq}{\end{equation}}
\newcommand{\dbi}[3]{#1 \dia_{#3} #2}
\def\zo{^{0,1}}
\def\oz{^{1,0}}
\def\lie#1#2{[#1,#2]}
\def\xto{\xrightarrow}
\def\hX{\Theta}
\def\GG{\g}
\def\fsgf{{F^*\oplus\GG\oplus F}}
\def\fx{x}

\let\cdefn=\c
\let\idefn=\i
\let\kdefn=\k
\let\Ldefn=\L
\long\def\redefineSymbols{
  \let\c=\cdefn
  \let\i=\idefn
  \let\k=\kdefn
  \let\L=\Ldefn
}

\newcommand\MG[1]{} 
\let\justifiedmarginpar=\marginpar
\renewcommand\marginpar[1]{\justifiedmarginpar[\raggedleft\footnotesize #1]{\raggedright\footnotesize #1}}
\newcommand\MGI[1]{} 
\newcommand\PX[1]{} 
\newcommand\MS[1]{} 
\newcommand\GGi[1]{} 
\newcommand{\noprint}[1]{}  

\newcommand\notneeded[1]{{\footnotesize #1}}

\usepackage[bookmarksnumbered,pdfpagelabels=true,plainpages=false,colorlinks=true,
            linkcolor=black,citecolor=black,urlcolor=blue]{hyperref}

  \DepRepTitle{Director of graduate studies}
  \DepRep{Alberto Bressan\\\math}

\def\math{Professor of Mathematics}
  \Advisor{Ping Xu\\\math}  \AdvTitle{Co-Chair of Committee}
  \MemberA{Mathieu Sti{\'e}non\\\math\\ Co-Chair of Committee}
  \MemberB{Martin Bojowald\\ Professor of Physics}
  \MemberC{Luen-Chau Li\\\math}
  \MemberD{Adrian Ocneanu\\\math\\[3ex]}  
\MemberE{Aissa Wade\\\math\\[3ex]}
\MemberF{Gr\'egory Ginot\\{\math} of Universit\'e Paris 6\\Special Signatory\\[3ex]}

\DeanTitle{Dean}
\Dean{Prof.\ Daniel J. Larson}

\begin{document}\redefineSymbols
\frontmatter

\Abstract{We introduce Courant algebroids, providing definitions, some historical 
notes, and some elementary properties.  
Next, we summarize basic properties of graded manifolds.  Then, drawing on the work by Roytenberg and others, we introduce the graded or supergraded 
language demostrating a cochain complex/ 
cohomology for (general) Courant algebroids.  
We review spectral sequences and show how this tool is used to 
compute cohomology for regular Courant algebroids with a split base.
Finally we analyze matched pairs of Courant algebroids including the complexified standard Courant algebroid of a complex manifold and the matched pair arising from a holomorphic Courant algebroid.
}

\Acknowledgements{I would like to thank my adviser, Ping Xu, for his continuous 
support during my PhD studies.  I also thank the staff of the Department of 
Mathematics at The Pennsylvania State University for their help with all the 
organizational matters surrounding my long-term studies.  Additionally I thank my colleagues and the professors in our working group, especially
Mathieu Sti{\'e}non, for numerous interesting, motivating and provocative 
discussions.  I also want to thank my collaborator, Gr{\'e}gory Ginot, for his continuous motivation.  Last but not least, I want to thank my parents, 
who encouraged me to enjoy the other things in life.
}

\makepreliminarypages
\singlespace
\mainmatter

\chapter{Introduction \MG{minimalistic}}

The skew-symmetric Courant bracket was first introduced by Courant in 1990 (see \cite{Cour90})
in order to describe Dirac manifolds, a generalization of presymplectic and
Poisson manifolds.  Parallel considerations in variational calculus led
Dorfman \cite{Dor87,Dor93} to derive a similar bracket which is not skew-symmetric,
but which fulfills a certain easy Jacobi identity.

Poisson manifolds are a slight generalization of symplectic
manifolds and have a rich geometric structure, starting from a Poisson
bracket, expanding to the association of a Hamiltonian vector field to a function, and further expanding to areas such as Poisson cohomology.  Presymplectic manifolds seem a bit odd in the beginning, because
they present no such nice structure, but nevertheless occur
naturally as a submanifold of a symplectic manifold.  This is
because the pullback of the symplectic form is closed under the de Rham
differential, yet may be degenerate.  The property of closedness defines a
presymplectic form.  

In the context of a twisted Courant bracket it is also possible to describe 
twisted Poisson structures as Dirac structures, as shown by \v{S}evera and 
Weinstein \cite{Sev01}.  These occur naturally in the context of gauge fixed 
Poisson sigma-models with boundary terms (see e.g.\ \cite{Str02}).

In 1997 Liu, Weinstein, and Xu introduced the notion of a Courant algebroid in 
order to describe Manin triples for Lie bialgebroids
.  This is the 
vector bundle analog of a Manin triple of Lie algebras.  Such a triple consists 
of a quadratic Lie algebra of split signature, a maximal isotropic Lie 
subalgebra, and its dual embedded into the quadratic Lie algebra.  The important 
property of the quadratic Lie algebra is that it has the additional structure 
of a symmetric nondegenerate bilinear form which is compatible with the Lie 
bracket.

In short, a Courant algebroid consists of a bracket on the sections of a
vector bundle (analogous to a Lie algebra) subject to a Jacobi identity,
an anchor map subject to a Leibniz rule, and an inner
product compatible with the bracket (analogous to quadratic Lie algebras).
However, due to the existing axiomatic conclusions, one has
to violate the skew-symmetry.  This version is called Dorfman
bracket.

With the skew-symmetrized bracket the Courant algebroid is an $L_\infty$
algebra of length 2, as observed by Roytenberg and Weinstein \cite{Royt98}.
An important step forward by Roytenberg \cite{Royt02} was a description of Courant
algebroids in terms of a derived bracket as introduced by Kosmann-Schwarzbach
\cite{YKS96}, or equivalently in terms of a nilpotent
odd operator (also known as $Q$-structure). Hence the Courant algebroid
structure with its intricate axioms can all be encoded in a cubic function
 $H$ on a graded symplectic manifold and its derived bracket.

\medskip

Recently, the complexified (exact) Courant algebroids have become popular as a
framework for the generalized complex geometry (see for example
\cite{Hit03,Gua04}).  
Moreover, a reduction procedure (analogous to Marsden--Weinstein reduction in
symplectic geometry \cite{Wei74}) of exact Courant algebroids and generalized
complex structures has been described by Bursztyn et al \cite{BCG07}.

\medskip

In addition to their use for Dirac structures, Courant algebroids permit access from 
the point of view of theoretical physics as target spaces of sigma models, see 
e.g\ \cite{Par00,Ikeda02,HoPa04,Royt06}. A long-term question is how 
to quantize Courant algebroids using the Feynman path integral. The first step 
has already been taken by Roytenberg who described the symmetries of the 
Courant sigma-model in AKSZ-BV formalism.

\bigskip

The plan of the thesis is as follows.  Chapter 2 presents Lie algebroids, Lie bialgebroids, Courant algebroids, and Dirac structures.  Moverover it introduces the graded and supergraded geometry.  Then we explain the derived bracket construction for Courant algebroids and describe Dirac structures in superlanguage.  We finish by explaining spectral sequences, which serve as the base of the constructions in Chapter 3.

Chapter 3 presents the detailed results and proofs computing the standard cohomology of regular Courant algebroids with a split base.  In particular, we prove a conjecture of Sti{\'e}non and Xu.

Chapter 4 gives detailed results and proofs about matched pairs of Courant algebroids.  Examples arise from complex geometry and certain regular Courant algebroids.

\chapter{Basic notions \MG{90\%}}
\label{s:vdef}
\section{Courant algebroids and Dirac structures}
In this chapter we will introduce Lie algebroids and Courant algebroids together
with some examples. Further, we give the definition and examples of Dirac structures. 


\subsection{Lie algebroids}
\subsubsection{Definition}
\begin{vdef}\index{Lie!algebroid, definition} A \emph{Lie algebroid} 
is a triple $(A\to M,[.,.],\rho)$ consisting of a vector bundle $A$ over a smooth manifold $M$, 
a skew-symmetric $\R$-bilinear operation $[.,.]:\Gamma(A)\otimes\Gamma(A)
\to\Gamma(A)$, where $\Gamma(A)$ are the smooth sections of $A$, and a vector bundle map $\rho\:A\to TM$, called the anchor, such that
\begin{align*}
  [\phi,[\psi_1,\psi_2]] &= [[\phi,\psi_1],\psi_2] +[\psi_1,[\phi,\psi_2]]
    \\
  [\phi,f\cdot\psi] &= \rho(\phi)[f]\cdot\psi +f\cdot[\phi,\psi]  
\end{align*}  where $\phi,\psi,\psi_1,\psi_2\in\Gamma(A)$ and $f\in\smooth(M)$.
\end{vdef}

The bracket is not associative, instead the first axiom (Jacobi-identity) 
tells us that $[\phi,.]$ is a derivation with respect to the bracket. This turns $\Gamma(A)$ into a Lie algebra.  The second axiom (Leibniz rule) tells us that the bracket is not $\smooth$-linear.

\begin{rem} A first observation due to Uchino \cite{Uchi02} is that the anchor map $\rho$ is 
a morphism of brackets, $\rho[\psi_1,\psi_2]=[\rho(\psi_1),\rho(\psi_2)]$, as 
it follows from the two axioms.  
\end{rem}

\subsubsection{Examples}
\begin{Exam} One class of examples are Lie algebras which are Lie algebroids over a point 
 $M=\{*\}$. Here $\rho=0$, $A=\g$ is a vector space with a bracket that has to 
fulfill the first axiom (Jacobi identity).
\end{Exam}

\begin{Exam} The tangent bundle $A=TM$ of a manifold $M$ 
is a Lie algebroid. Here $\rho=\id_{TM}$ and $[.,.]$ is the commutator of vector fields%
.\end{Exam}

\begin{Exam} 
Take an arbitrary vector
bundle $V\to M$ and consider the set of linear covariant differential-operators
(CDO).  These are $\R$-linear maps $\psi:\Gamma(V)\to \Gamma(V)$ such that there 
is a vector field $X\in\Xx(M)$, satisfying
 $$  \psi[f\cdot v] = X[f]\cdot v +f\cdot\psi[v] \;,
 $$  for all $v\in\Gamma(V), f\in\smooth(M)$.  

Note that given a fixed $\psi$ with an associated vector field $X$, then 
this $X$ is unique (as long as $V$ is of rank at least one).  
Moreover the set of all covariant 
differential operators forms a module over $\smooth(M)$ via 
\begin{align*}
  (f\cdot\psi)[g\cdot v]    &= (f\cdot X)[g]\cdot v +fg\cdot\psi[v]  \quad\forall f,g\in\smooth(M)\\
  (\psi+\phi)[g\cdot v]&= (X+Y)[g]\cdot v +g\cdot\psi[v]+g\cdot\phi[v] \quad\forall \phi,\psi\in\CDO
\intertext{where $Y$ is associated to $\phi$.\MG{I need $g$ here to show how to
 get the associated vector field.}  Therefore}
  \rho:\CDO&\to \vect(M) : \psi \mapsto X
\intertext{is a map of $\smooth(M)$-modules.  As any affine connection $\nabla$ 
on $V$ gives rise to a map,}
  \nabla\: \vect(M)&\to \CDO:\quad X\mapsto \nabla_X \;.\\
\intertext{It follows that $\rho$ is surjective.}
  \Gamma(\End(V)) &=\{\psi\in\CDO: \rho(\psi)=0 \} \;,
\intertext{hence, we obtain the following 
  short exact sequence}
  0\to &\Gamma(\End(V))\to \CDO\xrightarrow{\rho} \Xx(M)\to 0
\end{align*}
\end{Exam}
In fact, we have the following:
\begin{prop}[Mackenzie \cite{Mack05}]  Given a vector bundle $V\to M$, there is a vector 
  bundle $\DD(V)\to M$ whose sections are the covariant differential operators, 
  fitting into the short exact sequence.
\[  0\to\End(V)\to \DD(V)\xrightarrow{\rho}\ TM\to 0
\]
Moreover, $\DD(V)$ is a Lie algebroid with the commutator bracket.
\end{prop}

\begin{Exam}[Gauge Lie algebroid] Let $\g$ be a Lie algebra acting on a smooth 
manifold $M$, i.e.\ suppose we have a morphism $\bar\rho:\g\to \X(M)$ of Lie algebras:
\begin{align*}
  \bar\rho[X,Y] &= [\bar\rho(X),\bar\rho(Y)] \;.
\intertext{Then the bundle $M\times\g\to M$ can be endowed with the structure of a 
    Lie algebroid as follows:}
  \rho(X)(p) &:= \bar\rho(X(p)) \;.
\intertext{For the constant sections $X,Y\in\g$:}
  [X,Y](p) &= [X(p),Y(p)]
\intertext{and extended by Leibniz-rule to arbitrary sections as}
  [f\cdot X,g\cdot Y](p) &= f(p)g(p)\cdot[X(p),Y(p)] +f(p)\rho(X)(p)[g]\cdot Y(p)-g(p)\rho(Y)(p)[f]\cdot X(p)
\end{align*}
\end{Exam}

\begin{Exam} Any regular integrable distribution $D\subset TM$ is a Lie algebroid. The 
anchor map is the embedding $\rho:D\subset TM$ and the bracket is induced from 
the commutator bracket. \qed 
\end{Exam}

Yet another class of examples arises from Poisson geometry.  Recall that the Schouten-Nijenhuis bracket on multivector fields is the graded extension of the commutator bracket for tangent vector fields.  For simplicity of notation we use the same notation for both brackets.

\begin{vdef}  A \emph{Poisson manifold} $(M,\pi)$ is a smooth manifold endowed with
a bivector field $\pi\in\Gamma(\Lambda^2TM)$ subject to $[\pi,\pi]=0$.
\end{vdef}

There are two brackets arising from such Poisson structures:
\begin{prop}[Weinstein \cite{WeDress}
]  Given a Poisson bivector $\pi$, then
\begin{enumerate}\item  the smooth functions $\smooth(M)$ together with the 
skew-symmetric bracket
\begin{align*}
 \{f,g\} &=\pi(\ud f,\ud g) 
\intertext{form a Poisson algebra, i.e.}
 \{f,gh\} &= \{f,g\}h +g\{f,h\} \\
  \{f,\{g,h\}\} &= \{\{f,g\},h\} + \{g,\{f,h\}\} \;;
\end{align*}

\item  $T^*M$ is a Lie algebroid with the anchor map
\begin{align*}
  \rho(\alpha)[f] &= \<\pi,\alpha\wedge\ud f\> \;
\intertext{and whose Lie bracket fulfills}
  [\ud f,\ud g]_\pi &= \ud\{f,g\} \;, \\
  [\alpha, f\beta]_\pi &= f[\alpha,\beta]_\pi +\pi^\#(\alpha)[f]\cdot\beta \;.
\end{align*}
\end{enumerate}
\end{prop}

 For general 1-forms it reads as:
 $$ [\alpha,\beta]_\pi = \Lie_{\rho(\alpha)}\beta -\Lie_{\rho(\beta)}\alpha 
  -\ud(\pi(\alpha,\beta))  \quad\forall\alpha,\beta\in\Omega^1(M)
 $$
\begin{proof}  The Leibniz rule of the Poisson bracket follows from the Leibniz
rule of the de Rham differential.
\begin{align*}
  \{f,\{g,h\}\}&-\{\{f,g\},h\}-\{g,\{f,h\}\} \\
  &= [[\pi,f],[[\pi,g],h]]-[[\pi,[[\pi,f],g]],h] -[[\pi,g],[[\pi,f],h]]  \\
  &= [[[\pi,f],[\pi,g]],h] -[[\pi,[[\pi,f],g]],h]  \\
  &= [[[\pi,[\pi,f]],g],h]  \\
  &= \frac12[[[[\pi,\pi],f],g],h] =0
\intertext{Jacobi identity for exact 1-forms follows from this.  It remains to 
    check that the Jacobi identity holds when the forms are multiplied by functions.  We show this by showing that the Jacobiator is a tensor:}
  Jac(\alpha,\beta,\gamma) &:= [\alpha,[\beta,\gamma]_\pi]_\pi +[\gamma,[\alpha,\beta]_\pi]_\pi +[\beta,[\gamma,\alpha]_\pi]_\pi \\
  Jac(f\alpha,\beta,\gamma) &= -\rho[\beta,\gamma]_\pi[f]\cdot\alpha +f[\alpha,[\beta,\gamma]_\pi]_\pi -[\gamma,\rho(\beta)[f]\cdot\alpha]_\pi \\
  &\quad +[\gamma,f[\alpha,\beta]_\pi]_\pi +[\beta,\rho(\gamma)[f]\cdot\alpha]_\pi +[\beta,f[\gamma,\alpha]_\pi]_\pi \\
  &= f\,Jac(\alpha,\beta,\gamma)
\end{align*}
\end{proof}

\subsection{Lie bialgebroids}
Given a Lie algebroid $A\to M$, we can define a differential
 $\ud_A:\Omega^k_M(A)\to\Omega^{k+1}_M(A)$ with
 $\Omega_M^\bullet(A):=\Gamma(\Lambda^\bullet A^*)$ as follows:
\begin{align}  (\ud_A\omega)(\psi_0,\ldots,\psi_n) = 
  &\,\sum_{i=0}^n (-1)^i\rho(\psi_i)[\omega(\psi_1,\ldots,
      \widehat{\psi_i},\ldots,\psi_n)]  \nonumber\\
  &+\sum_{i<j} (-1)^{i+j}\omega([\psi_i,\psi_j],\psi_1,\ldots,
      \hat\psi_i,\ldots,\hat\psi_j,\ldots,\psi_n)  \label{dLie}
\end{align}  where $\psi_i\in\Gamma(A)$.

\begin{vdef}  A \emph{Lie bialgebroid} $(A,A^*)$ is a pair of Lie algebroids whose 
vector bundles are dual to each other and whose structures are compatible in the sense that
 $$ \ud_{A^*}[\phi,\psi]_A = [\ud_{A^*}\phi,\psi]_A +[\phi,\ud_{A^*}\psi]_A
  \quad\forall\phi,\psi\in\Gamma(A) $$
\end{vdef}

\begin{Exam}  A Lie bialgebra is a Lie algebra $(\g,[.,.])$ together with a differential $\mu:\g\to\Lambda^2\g$ which is a 1-cocycle $\delta\mu=0$, where $\delta$ is the Chevalley--Eilenberg differential from Lie algebra cohomology.  The differential $\mu$ gives rise to a bracket $l_2=[.,.]_{\g^*}$ on the dual space.  Due to duality this bracket fulfills $[\delta,l_2]=0$, which is the defining property of a Lie algebra. A Lie bialgebra arises as infinitesimal version of a Poisson--Lie group (see also \cite{Drin86}).  Due to a proposition by De Smedt \cite{Sme94} every Lie algebra permits a non-trivial bialgebra structure.  For further details see, e.g.\ \cite{LW90,Gutt93}.
\end{Exam}
\begin{Exam} More particular examples are exact Lie bialgebras.  A Lie bialgebra $(\g,[.,.],\mu)$ is called exact iff $\mu=\delta r$.  In this case $\delta\mu=0$ follows trivially.  In order for $\mu$ to be nil-potent $r\in\Lambda^2\g$ needs to fulfill:
 $$ \forall X\in\g: [[r,r],X]=0 $$  One particular solution is $[r,r]=0$ where again $[.,.]$ is the Lie bracket extended to exterior powers of the Lie algebra.
 
 A particular example on $\sl_2\C$ is as follows.  We choose the basis $\{H,X_+,X_-\}$ with $[H,X_\pm]=\pm2X_\pm$ and $[X_+,X_-]=H$.  One choice of $r$ is $r=H\wedge X_+$, because here $[r,r]=0$ and $(\delta r)(X_-)=2X_-\wedge X_+\ne0$.  The other two derivations are $(\delta r)(X_+)=0$ and $(\delta r)(H)=2H\wedge X_+$.  Therefore the Lie algebra structure on the dual space $\g^*$ in the dual base $\{H^*,X_+^*,X_-^*\}$ is $[H^*,X_+^*]_*=\frac12H^*$, $[X_-^*,X_+^*]_*=\frac12X_-^*$, and $[H^*,X_-^*]_*=0$.
\end{Exam}

\begin{prop}[Mackenzie and Xu \cite{Mack94}] Every Poisson manifold gives rise to a Lie bialgebroid $((T^*M)_\pi,TM)$.
\end{prop}
\begin{proof}  We have to show
\begin{align*}
  \ud[\alpha,\beta]_\pi &= [\ud\alpha,\beta]_\pi +[\alpha,\ud\beta]_\pi
  \quad\forall\alpha,\beta\in\Omega^1(M) \;.
\intertext{The formula is trivial for exact forms.  For the general case we 
    have to check that both sides gain the same terms when multiplying a form with
    a smooth function.}
  \ud[\ud f,g_1\ud g_2] &= \ud(\rho(\ud f)[g_1]\ud g_2) +\ud g_1\wedge[\ud f,\ud g_2]_\pi  \\
  &= [\ud f,\ud g_1]_\pi\wedge\ud g_2 +\ud g_1\wedge[\ud f,\ud g_2]_\pi \\
  &= [\ud^2 f,g_1\ud g_2]_\pi +[\ud f,\ud(g_1\ud g_2)]
\end{align*}
{The same formula holds for the first argument of the bracket due to 
  skew-symmetry.}
\end{proof}

\subsection{Courant algebroids}
\subsubsection{Definition}
\begin{vdef}\label{d:Courant}\index{Courant!Definition} A \emph{Courant algebroid}
  is a vector bundle $E\to M$ with 
  three operations: a symmetric non-degenerate bilinear form
  $\langle.,.\rangle$ on $E$, an $\R$-bilinear operation 
  $\dia:\Gamma(E)\otimes\Gamma(E)\to\Gamma(E)$ on 
  the sections of $E$ and a bundle map $\rho:E\to TM$, called the anchor map,  
  satisfying the following rules:
  \begin{align}
  \dor{\phi}{(\dor{\psi_1}{\psi_2})} &= \dor{(\dor{\phi}{\psi_1})}{\psi_2} 
    +\dor{\psi_1}{(\dor{\phi}{\psi_2})}  \label{ax:Jacobi}  \\
  \dor{\phi}{(f\cdot \psi)} &= \rho(\phi)[f]\cdot\psi +f\cdot(\dor{\phi}{\psi})  \label{ax:Leibn}\\
    \dor{\phi}{\phi} &=\frac12\rho^*\ud\langle\phi,\phi\rangle  \label{ax:nSkew}\\
    \rho(\phi)\langle\psi,\psi\rangle &= 2\langle\dor{\phi}{\psi},
      \psi\rangle  \label{ax:adInv}
  \end{align} where $\phi,\psi,\psi_1,\psi_2$ are sections of $E$, $f$ is a 
  smooth function on $M$, $\ud$ is the de Rham differential (from functions
   to 1-forms) and
   $$\rho^*:T^*M\to E $$ is induced by $\rho$ and the identification of $E$ and 
   $E^*$ via the bilinear form $\<.,.\>$.
\end{vdef}

The first two axioms are similar to those of a Lie algebroid and imply in particular
\beq{rhomor} \rho(\phi\dia\psi) = [\rho\phi,\rho\psi] \eeq
as observed by Uchino \cite{Uchi02}.  The last 
axiom is the invariance of the inner product under the bracket and equivalent 
to
 $$ \rho(\phi)\<\psi_1,\psi_2\> = \<\dor{\phi}{\psi_1},\psi_2\> +\<\psi_1,
  \dor{\phi}{\psi_2}\> \;.$$

\begin{Exam} Liu et al \cite{Xu97}.  Let $(A,A^*)$ be a Lie bialgebroid and consider the
direct sum $A\oplus A^*$ with the natural inner product and the bracket:
\begin{align*} 
  \<X\oplus\alpha,Y\oplus\beta\> &= \alpha(X)+\beta(Y)  \\
  \rho(X\oplus\alpha) &= \rho_A(X)+\rho_{A^*}(\alpha)  \\
  \dor{(X\oplus\alpha)}{(Y\oplus\beta)} &= 
  \left([X,Y]_{A} +\Lie^{A^*}_\alpha Y-\imath_\beta\ud_{A^*}X \right) 
  \oplus \left( [\alpha,\beta]_{A^*}+\Lie_X\beta -\imath_Y\ud_A\alpha \right)
\end{align*}  
Then this forms a Courant algebroid.
\end{Exam}

However, the operation $\dia$ is not skew-symmetric, but equivalent to a 
skew-symmetrized one.

\begin{prop}  The non-skew-symmetric Courant bracket with its axioms is equivalent to the 
skew-symmetric bracket $[\phi,\psi]:=\frac12(\dor{\phi}{\psi}-\dor{\psi}{\phi})$ 
and the following axioms
\begin{align*}
  [[\psi_1,\psi_2],\psi_3] &+[[\psi_2,\psi_3],\psi_1]
    +[[\psi_3,\psi_1],\psi_2] 
    = \tfrac16\D \left(\langle[\psi_1,\psi_2],\psi_3\rangle+\text{cycl.}\right)  \\
  [\psi_1,f\cdot\psi_2] &= f\cdot[\psi_1,\psi_2] +\rho(\psi_1)[f]\cdot\psi_2 
    -\tfrac12\<\psi_1,\psi_2\>\cdot\D f  \\
  \<\D f,\D g\> &= 0 \\
  \rho(\phi)\<\psi_1,\psi_2\> &= 
    \big\langle[\phi,\psi_1]+\tfrac12\D\<\phi,\psi_1\>,\psi_2\big\rangle 
    +\big\langle\psi_1,[\phi,\psi_2]+\tfrac12\D\<\phi,\psi_2\>\big\rangle
\end{align*}
The non-skew-symmetric bracket can be reconstructed as
$$ \dor{\phi}{\psi} = [\phi,\psi]+\frac12\D\<\phi,\psi\>
$$

Where 
\begin{align*}
  \D&:\smooth(M)\to\Gamma(E) \\
  \<A,\D f\> &= \rho(A)[f]
\end{align*}
\end{prop}
\noindent The proof can be found in Roytenberg \cite[Prop.2.6.5]{Royt99}.

\begin{rem}Analogously to Lie algebroids the anchor map is a morphism of brackets, 
i.e.\ $\rho(\dor{\phi}{\psi})=[\rho(\phi),\rho(\psi)]$. As observed by 
K.~Uchino \cite{Uchi02} this follows from the first 2 axioms, for Lie algebroids 
as well as for Courant algebroids (also for the skew-symmetric bracket).
\end{rem}

\begin{rem}  Courant algebroids produce examples of the following definition (see also \cite[defn.p.1094]{SHLA}).

  An $L_\infty$-algebra is a graded vector space $V=\oplus_{i\in\Z}V_i$ together with operations: $$ l_i:\Lambda^i V\to V $$ of degree $2-i$.  The $l_i$ are extended as derivations to $\Lambda^\bullet V\to\Lambda^{\bullet-i+1}V$ which by abuse of notation we write with the same symbols.  The $l_i$ are subject to the axioms:
\begin{align*}
  \sum_{i+j=n+1} (-1)^{i(j+1)} (l_i\circ l_j)(e_1,\ldots,e_n)=0 \quad\forall e_i\in V, n\ge1
\end{align*}

Let $V_\bullet$ be the graded vector space $V_0=\Gamma(E),
  V_{1}=\smooth(M), V_2=\ker(\rho^*\ud)$, $V_n=0$ for $n\ge3$, together with 
  the operations
\begin{align*}
  l_1(f) =& \begin{cases} \rho^*\ud f &\text{for } f\in V_1  \\
    f &\text{for } f\in V_2 \\
    0 &\text{otherwise}  \end{cases}  \\
  l_2(e_1,e_2) =& \begin{cases} [\phi_1,\phi_2]
     &\text{for } e_i=\phi_i\in V_0 \\
    \rho(\phi)[f] &\text{for } e_1=\phi\in V_0, e_2=f\in V_1 \\
    -\rho(\phi)[f] &\text{for } e_1=f\in V_1, e_2=\phi\in V_0 \\
    0 &\text{otherwise} \end{cases}  \\
  l_3(\phi_1,\phi_2,\phi_3) =& \begin{cases}
    -\frac13\langle[\phi_1,\phi_2],\phi_3\rangle-\text{cycl.}
     &\text{for } \phi_i\in V_0 \\
    0 &\text{otherwise}  \end{cases}\\
  l_n =&\, 0\quad\text{for }n\ge4 \;.
\end{align*}  

The theorem now reads as follows \cite{Royt98}: The above structure makes a Courant algebroid an $L_\infty$-algebra.
\end{rem}

\subsubsection{Exact Courant algebroids}
\begin{vdef} A Courant algebroid $E\to M$ is called \emph{transitive} iff its anchor map 
is surjective.  It is called \emph{exact} iff in addition $$ \rk E=2\dim M\;. $$
\end{vdef}  
Exact Courant algebroids fit into the short exact sequence:
 $$ 0\to T^*M\xrightarrow{\rho^*} E\xrightarrow{\rho}TM\to 0 \;, $$
because 
 $$\rho\circ\rho^*\ud\<\psi_1,\psi_2\>=\rho(\psi_1\dia\psi_2)+\rho(\psi_2\dia\psi_1)
  =[\rho\psi_1,\rho\psi_2]+[\rho\psi_2,\rho\psi_1]=0$$ and $\<\psi_1,\psi_2\>$ 
locally generates all functions.  Exactness is due to dimensional reasons.
An isotropic splitting of this sequence is $\sigma:TM\to E$ such that
 $\rho\circ\sigma=\id_{TM}$ and $\<\sigma(X),\sigma(Y)\>=0$ for all $X,Y\in TM$.  
These exist, because $E$ is a vector bundle.

If we choose an isotropic splitting $\sigma:TM\to E$, we can introduce the 
function $H(X,Y,Z):=\<\dor{\sigma(X)}{\sigma(Y)},\sigma(Z)\>$.  Then $H$ is a closed 3-form.  Under the change of the splitting $\sigma\mapsto\sigma+B^\sharp$ for any 2-form $B$, $H$ changes by the exact term $\ud B$.

\begin{prop}[{\v S}evera \cite{SevLett}]   There is a one-to-one correspondence between equivalence classes of exact Courant algebroids and $H^3_{dR}(M)$.
\end{prop}
\begin{proof}  The above construction defines the closed 3-form for a given exact Courant algebroid.  Conversely given a closed 3-form the following example constructs a corresponding exact Courant algebroid.  A short proof that this fulfills indeed the axioms of a Courant algebroid is given at the end of section \ref{s:derived}.
\end{proof}

\begin{Exam}\textbf{ twisted Courant algebroid \cite{SevLett}.}\label{ex:Sev} $E=TM\oplus T^*M$ with
the inner product $$ \<X\oplus\alpha,Y\oplus\beta\> = \alpha(Y)+\beta(X) \;,$$
and the anchor map
 $$\rho(X\oplus\alpha)=X\;.$$  
Let $H_{twist}\in\Omega^3(M)$ be a closed 3-form.  One can now define a 
bracket as:
\begin{equation}  \dor{(X\oplus\alpha)}{(Y\oplus\beta)} := [X,Y]\oplus 
  \left(\Lie_X\beta-\imath_Y\ud\alpha +\imath_X\imath_Y H_{twist}\right)  \label{Cbrack}
\end{equation} For $X,Y\in\vect(M)$, $\alpha,\beta\in\Omega^1(M)$.
It is called the (twisted) Dorfman bracket.
This bracket forms a Courant algebroid.  The skew-symmetric version reads as
 $$ [X\oplus\alpha, Y\oplus\beta] := [X,Y]\oplus \left(\Lie_X\beta
-\Lie_Y\alpha +\frac12\D(\alpha(Y)-\beta(X))+\imath_X\imath_Y H_{twist}\right)
 $$
\end{Exam}

\begin{Exam} More generally, start with an arbitrary Lie
algebroid $A$ (and $A^*$ its dual endowed with differential $\ud_A$).  Pick a 
 $\ud_A$-closed $A$-3-form $H_{twist}\in\Gamma(\Lambda^3A^*)$ and define inner product, anchor map and 
the operation analogously on $A\oplus A^*$.
\begin{align*}
  \<X\oplus\alpha,Y\oplus\beta\> &= \alpha(Y) +\beta(X)  \\
  \rho(X\oplus\alpha) &= \rho_A(X)  \\
  \dor{(X\oplus\alpha)}{(Y\oplus\beta)} &= [X,Y]_A \oplus 
    \left(\Lie_X\beta -\imath_Y\ud_A\alpha +\imath_X\imath_Y H_{twist}\right) \;,
\end{align*}  then this forms a Courant algebroid.  With the $A$-Lie derivative
 $\Lie_X=[\imath_X,\ud_A]$.
\end{Exam}
\bigskip

\subsection{Dirac structures \MG{90\%} I}
\begin{vdef}\label{d:induced}  Given a Courant algebroid $(E\to M,\dia)$ and a vector subbundle $L\subset E$ with support on a submanifold $N\subset M$, we say that an operation 
$\dia_L:\Gamma_N(L)\otimes\Gamma_N(L)\to \Gamma_N(E)$ is \emph{induced} by $\dia$ on
 $E$ iff all $\psi_i\in\Gamma_N(L)$, $\psi_i'\in\Gamma(E)$ with
  $\psi_i=\psi_i'|_N$ fulfill $\psi_1\dia_L\psi_2=(\psi_1'\dia\psi_2')|_N$.
\end{vdef}

The following slightly generalized definition of Dirac structures goes back
to Brusztyn, Iglesias-Ponte and {\v S}evera \cite{BIS08}.

\begin{vdef}\label{d:Dirac} Given a Courant algebroid $E\to M$. A subbundle
  $\begin{CD}L @>{\varsubset}>> E\\
    @VVV  @VVV \\
    N @>{\varsubset}>>  M\end{CD}$ ($N\subset M$ a submanifold) is called
\begin{description}\item[Lagrangian] iff it is isotropic and has rank half the rank of $E$;
\item[almost Dirac] iff it is isotropic and the anchor maps $L$ into $TN$;
\item[Dirac]\index{Dirac structure|strong} iff it is Lagrangian, sub-Dirac, and for all sections $\psi_i\in\Gamma(E)$ with $\psi_i|_N\in\Gamma(L)$, $(\psi_1\dia\psi_2)|_N\in\Gamma(L)$.
\end{description}
\end{vdef}

\begin{prop}\begin{enumerate}\item  Given a 2-form $\omega\in\Omega^2(M)$ then its graph 
 $$\braph(\omega):=\{X\oplus \imath_X\omega: X\in TM\}\subset (TM\oplus T^*M)$$
is Dirac in the standard Courant algebroid iff $\ud\omega=0$.

\item  The graph of a bivector $\pi\in\Gamma(\Lambda^2TM)$, 
 $$\braph(\pi)=\{\pi^\#\alpha\oplus\alpha: \alpha\in T^*M\}$$
is Dirac in the standard Courant algebroid iff it is Poisson, i.e.\ $[\pi,\pi]=0$.
\end{enumerate}
\end{prop}

Therefore Dirac structures (in the standard Courant algebroid) unify the notion 
of (pre)\-symplectic and Poisson structures on the manifold.\par\bigskip
Next we wish to investigate under which conditions an isotropic submanifold
of a Courant algebroid inherits a bracket.

Note that a generalized Dirac structure over an arbitrary submanifold does not necessarily inherit a bracket.  The condition for this is as follows.
\begin{prop}\label{p:induced}  A vector subbundle $\begin{CD}L @>\varsubset>> E\\ 
  @VVV @VVV\\N @>\varsubset>> M \end{CD}$ of a Courant algebroid 
 $(E,\langle.,.\rangle,\dia,\rho)$ has an induced operation 
 $\dia_L:\Gamma(L)\times\Gamma(L)\to\Gamma_N(E)$ iff $\rho(E|_N)\subset TN$.
\end{prop}
\MG{shorten!}
\begin{proof}$\quad$\begin{steps}\step[denseness of extensible sections]  Since $N$
  is an embedded submanifold the functions on $M$ restricted to $N$ are
  dense in the set of all smooth functions on $N$.  Analogously, for every point
  $p\in N$ there is an open neighborhood $U\subset M$ over which $E$ and
  $L_{U\cap N}$ are trivial and $E_U$ has a frame $\{e_a:1\le a\le k\}$
  ($k:=\rk E$) such that $\{e_a|_N:1\le a\le l\}$ ($l:=\rk L$) form a
  frame of $L$ over $U\cap N$.

  Therefore a bracket that is defined on extensible sections uniquely extends
  to all (smooth) sections of $L$.
\step[ambiguity of the extension]  The question is now whether the bracket
depends on the extension of a section.  Let $\psi_i\in\Gamma(E)$ such that
$\psi_i|_N\in\Gamma(L)$.  Then $\psi_1+\widetilde{\psi_1}$ would be another
extension of $\psi_1|_N$ iff $\widetilde{\psi_1}\in\Gamma(E)$ and
$\widetilde{\psi_1}|_N=0$.  It turns out that the set $I(N)\cdot\Gamma(E)$, where
$I(N):=\{f\in\smooth(M):f|_N=0\}$, is dense in the set of all sections
$\widetilde{\psi_1}$.  Therefore one needs to check whether
\begin{align*}
  (\dor{\psi_1}{\psi_2})|_N &= (\dor{(\psi_1+f_1\cdot\psi_1')}{(\psi_2+f_2\cdot\psi_2')})|_N  \\
\intertext{or equivalently}
  0=& (\dor{\psi_1}{(f_2\cdot\psi_2')}+\dor{(f_1\cdot\psi_1')}{\psi_2} +\dor{(f_1\cdot\psi_1')}{(f_2\cdot\psi_2')})|_N  \\
  =& f_2\cdot(\dor{\psi_1}{\psi_2'})+\rho(\psi_1)[f_2]\cdot\psi_2' 
  +f_1\cdot(\dor{\psi_1'}{\psi_2})-\rho(\psi_2)[f_1]\cdot\psi_1' \\ 
  &+\langle\psi_1',\psi_2\rangle \D f_1
  +f_2\dor{(f_1\cdot\psi_1')}{\psi_2}+f_1\rho(\psi_1')[f_2]\cdot\psi_2' \quad\text{all restricted to }N
\end{align*}
The terms containing at least 1 factor $f$ obviously vanish on $N$.  
\step[``$\Leftarrow$''] Also the terms containing the factors $\rho(\psi)[f]$
vanish on $N$ because of the following fact from differential geometry:
  \begin{lemma}\label{l:emb}  Let $N\subset M$ be an embedded submanifold.  Then
    $X\in\X(M)$ which has $X|_N\in\X(N)$  maps $I(N)\to I(N)$.
  \end{lemma}  Here $I(N)$ denotes the ideal of functions vanishing on the submanifold $N\subset M$.
  The last term vanishes if one can prove that $\D f|_N=0$.
  $\D=\rho^*\circ\ud_M|_N$  gives $(\D
  f)|_N=\rho^*\circ(\tilde\ud_N\oplus\ud_N)f$  where $\ud_N$
  is the de Rham differential on $N$ and
  $\tilde{\ud_N}:I(N)\to\Gamma(N^*N)$ is the second part of the de Rham
  differential of $M$ when the image is restricted to $N$.  More precisely
  $0\to N^*N\xrightarrow{e} T^*_NM\to T^*N\to0$.

  Clearly the second part vanishes since $f|_N=0$, but the first part
  requires $\rho^*|_{N^*N}=0$.  This condition is also necessary if
  the map $\D_L:\smooth(N)\to\Gamma_N(E)$ should be induced by the map
  $\D:\smooth(M)\to\Gamma(E)$.

  Finally considering the dual $0\leftarrow NN\xleftarrow{e^*} T_NM\leftarrow
  TN\leftarrow0$  then $\rho^*|_{N^*N}=0$ is equivalent to $\rho^*\circ e=0$
  and equivalent to $e^*\circ\rho=0$.  But this is equivalent to 
  $\rho(E|_N)\subset TN$.

\step[``$\Rightarrow$'']  From the previous analysis it follows that all the
remaining terms beside those containing a factor from $I(N)$ have to vanish
together.  However setting $f_1$ to 0 one sees that $\rho(L)\subset TN$ is needed
first (which is therefore an explicit condition in the definition of
sub-Dirac structure).  

But then again the term
 $\langle\psi_1',\psi_2\rangle\D f_1$ is the remaining one.  From the
well-definedness of $\dia_L$ one concludes that this has to vanish.  The inner
product takes values in a dense set of $\smooth(M)\cap\smooth(N)$, thus $(\D
f)|_N$ has to vanish.
Since the last part in the previous step was an iff one can conversely
conclude that actually $\rho(E|_N)\subset TN$.
\end{steps}
\end{proof}

An immediate consequence is
\begin{cor}\label{c:diracisLie}  If the bracket is well defined on an isotropic 
  subbundle $L\to N$ of a Courant algebroid $E\to M$ it forms automatically a Lie 
  algebroid.
\end{cor}
\begin{proof}  Extend the sections isotropically and observe that the bracket
  is skew-symmetric.
\end{proof}

\section{Graded geometry} \label{app:graded}
Here we give a short introduction to graded geometry, graded vector bundles, and graded symplectic and Poisson manifolds.  We then continue with the derived bracket and standard cohomology of Courant algebroids.  We finish with description of Dirac structures in superlanguage.

\subsection{Comparison to supermanifolds}
For a good and detailed introduction to supermanifolds see, for
instance,~\cite{Vor91}.  Further references for graded manifolds are \cite[section 4]{Vor01}, \cite[section 2]{Sev01b} or \cite[section 2]{Royt02}.

First we need the notion of a ringed space/ manifold.
\begin{vdef} A ringed space $(M,\gsmooth)$ is a topological (Hausdorff) space $M$ together with a sheaf of rings $\gsmooth(\bullet)$.

A ringed manifold $(M,\gsmooth)$ is a ringed space $(M,\gsmooth)$ such that $\smooth_M$ is a subsheaf of $\gsmooth$.
\end{vdef}

Every smooth manifold is canonically a ringed space, here $\gsmooth=\smooth$.  Note that in this case the (sheaf of) ring(s) separate points.

\begin{vdef}\index{supermanifold} A \emph{supermanifold} is a ringed manifold $(M,\ssmooth)$ of $\Z_2$-graded algebras $\ssmooth(\bullet)$ which are locally free and of constant dimension, i.e.\ there is a $p\in\NN$ such that for every point $x\in M$, there is an open neighborhood $U\subset M$ with
\[  \ssmooth(U)\cong \smooth(U)\otimes \Lambda^\bullet\R^p
\]

  A smooth map $f\:(M,\ssmooth)\to(N,\P)$ of supermanifolds is a map of base manifolds $f\:M\to N$ together with a map of sheaves $f^*\:\P\to\ssmooth$ which is a morphism of $\Z_2$-graded algebras.  It is said to be injective iff $f^*$ is surjective.  It is said to be surjective iff $f^*$ is injective.
\end{vdef}
\begin{Exam}\begin{enumerate}\item Let $E\to M$ be a vector bundle.  We construct the supermanifold $\Pi E$ as space $M$ with the structure sheaf $\ssmooth(U):=\Gamma(U,\Lambda^\bullet E^*)$ for $U\subset M$, open.

\item let $\M:=U\times\Pi\R^p$ with $U\subset\R^q$ be a supermanifold. We introduce the coordinates $x^i$ on $U$, $\xi^a$ on the odd fiber $\Pi\R^p$.  One example of a smooth map is the following $f\:\M\to\M$ with
\begin{align*}
  f|_M &= \id_M, \\
  f^*(x^i) &= x^i+\frac12R^{\tilde{\imath}}_{ab}(x)\xi^a\xi^b \\
  f^*(\xi^a) &= M^{\tilde a}_b(x)\xi^b +\frac16S^{\tilde{a}}_{bcd}(x)\xi^b\xi^c\xi^d
\end{align*} for smooth functions $R$, $M$, $S$.
\end{enumerate}
\end{Exam}
The notation $\Lambda^\bullet \R^p$ means the exterior algebra of the vector space $\R^p$.  This means an algebra generated by $p$ generators $\xi_1,\ldots,\xi_p$ under the relation $\xi_j\xi_i=-\xi_i\xi_j$ for $1\le i,j\le p$.

\begin{vdef}\index{graded manifold} A \emph{graded manifold} is a ringed manifold $(M,\gsmooth)$ of $\Z$-graded algebras $\gsmooth(\bullet)$ which are locally free and of constant dimension, i.e.\ there are non negative integers $p_i$ such that for every point $p\in M$, there is an open neighborhood $U\subset M$ with
\[  \gsmooth(U)\cong {S^\bullet/\Lambda^\bullet}\R^{p_{-k}}\otimes\ldots\otimes \Lambda^\bullet \R^{p_{-1}}\otimes
  \smooth(U)\otimes\Lambda^\bullet\R^{p_1} \otimes S^\bullet\R^{p_2} \otimes\Lambda^\bullet\R^{p_3} \otimes\ldots\R^{p_n}
\]  The product on the right hand side implies that $\R^{p_1}$ are the generators of degree 1, $\R^{p_2}$ the generators of degree 2, \ldots .

A smooth map $f\:(M,\gsmooth)\to(N,\P)$ of graded manifolds is a smooth map of the bases $f\:M\to N$ together with a contravariant map $f^*\:\P\to\gsmooth$ where $f^*$ is a morphism of $\Z$-graded algebras.  It is said to be injective iff $f^*$ is surjective.  It is said to be surjective iff $f^*$ is injective.

Given a smooth map $f\:(M,\ssmooth)\to(N,\P)$, the inverse image for an $U\subset N$ is the ``manifold'' $f^{-1}(U)\subset M$ together with the restricted structure sheaf $\ssmooth|_{f^{-1}(U)}$.
\end{vdef}

Here $S^\bullet \R^{p_2}$ means the symmetric (i.e.\ poynomial) algebra in $p_2$ generators $x_1,\ldots,x_{p_2}$.  The notation $S^\bullet/\Lambda^\bullet$ means that in the lowest degree $k$ you have to take the symmetric/ exterior algebra depending on whether $k$ is even/ odd.

Note that in the above notation with $p_n\ne0$ we call $n$ the degree of the graded manifold $(M,\gsmooth)$.

Let $\dim M=p_0$, then we call $\dim(M,\gsmooth)=(\ldots,p_{-1},p_0,p_1,p_2,\ldots)$ the dimension of $\M=(M,\gsmooth)$.  If $(M,\gsmooth)$ has nonzero dimension only in non-negative degrees we call it an $\NN$-manifold.

\begin{Exam}\begin{enumerate}\item\index{graded manifold!example} Let $E\to M$ be a vector bundle and $n$ an integer. The graded manifold $E[n]$ is the manifold $M$ together with the structure sheaf
 $$ \gsmooth(U):= \Gamma(U,S^\bullet E^*[n]) \;. $$  Here and in what follows we use the compact notation $S^\bullet V$ for the graded symmetric tensor algebra of a graded vector space.  In the case $V=V_1[1]$, i.e.\ generators in odd degree, this means the exterior algebra and in the case $V=V_2[2]$, i.e.\ generators of even degree, this means the algebra of polynoms.

\item\index{graded manifold!example} Let $E\to M$ be a $\Z$-graded vector bundle, i.e.\ the fiber is a $\Z$-graded vector space $F\cong \R^{p_1}\oplus\R^{p_2}\oplus\ldots\R^{p_n}$ and the transition functions preserve the grading.  We define the graded manifold $\E$ as the manifold $M$ together with the structure sheaf:
\[ \gsmooth(U):= \Gamma(U, S^\bullet E^*) \]  

\item\index{graded manifold!example} Let $\E=(M,\gsmooth)$ and $\F=(N,\mathcal{P})$ be graded manifolds.  Their cartesian product $\E\times\F$ is the graded manifold over $M\times N$ with structure sheaf $\gsmooth\boxtimes\mathcal{P}$.  Note that the tensor product of sheaves locally fulfills
 $$ (\gsmooth\boxtimes\mathcal{P})(U\times V) \cong \ldots\otimes\Lambda^\bullet\R^{p_{-1}+q_{-1}}\otimes\smooth(U\times V)\otimes \Lambda^\bullet \R^{p_1+q_1}\otimes S^\bullet \R^{p_2+q_2}\otimes\ldots \;. $$

\item\index{graded manifold!example} Let $\M:=U\times\R^{p_1}[1]\times\R^{p_2}[2]$ be a graded manifold.  We introduce the coordinates $x^i$ on $U$, $\xi^a$ of degree 1 on $\R^{p_1}[1]$, and $b^B$ of degree 2 on $\R^{p_2}[2]$.  A particular example of a smooth map is $f\:\M\to\M$ with
\begin{align*}
  f|_M &= \id_M \\
  f^*(x^i)   &= \tilde{x}^i(x) \\
  f^*(\xi^a) &= M^{\tilde{a}}_b(x)\xi^b \\
  f^*(b^B)   &= N^{\tilde{B}}_C(x) b^C +\frac12R^{\tilde B}_{ab}(x)\xi^a\xi^b \;.
\end{align*}  This is the most general smooth map on this graded manifold.  Note that due to the $\Z$-grading and the absence of negative degrees the new coordinates of degree $k$ depend only on coordinates of degree less or equal $k$.
\end{enumerate}
\end{Exam}

For supermanifolds we have the following result.
\begin{thm}[Batchelor] Given a supermanifold $(M,\ssmooth)$ there exists a vector bundle $E\to M$ such that $(M,\gsmooth)\cong\Pi E$.
\end{thm}  For a proof, see e.g.\ \cite{Vor91}.

In this way the functor $\Pi\:\Vect\to\Smfd$, from the category of (ordinary) vector bundles to the category of supermanifolds, is surjective on objects.  Note however that the category of supermanifolds has more morphisms.

Remember, an $\NN$-manifold is a graded manifold where the coordinates have only non-negative degree.  For such manifolds we have the following proposition
\begin{prop}[Roytenberg \cite{Royt02}]  Given an $\NN$-manifold $\M=(M,\gsmooth)$ there is a tower of fibrations $\M_n=\M\to\M_{n-1}\to\ldots\to\M_1\to\M_0=M$, where the $\M_k$ are graded manifolds of degree at most $k$ over $M$. The bundles $\M_{k+1}\to\M_k$ are affine bundles.
\end{prop} 
\begin{proof}  The sheaves $\gsmooth_k$ are locally generated by the generators of $\gsmooth$ up to degree $k$.  We have an embedding of sheaves $\gsmooth_k\to\gsmooth_{k+1}$ which gives a projection of graded spaces $\M_{k+1}\to\M_k$. The fibers are affine spaces with fiber the generators of degree $k+1$.
\end{proof}
Note that the category of $\NN$-manifolds of degree 1, $\Nmfd_1$, coincides with the category of (ordinary) vector bundles $\Vect$,
 $$ \Nmfd_1\cong\Vect \;. $$

\subsection{Graded vector bundles, tangent and cotangent bundle}
We give a definition of graded vector bundles and proceed to the definition of tangent and cotangent bundles.  The latter will be an important example for graded symplectic manifolds.
\begin{vdef} A \emph{graded fiber bundle} is a projection $p:\E\to\M$ such that locally 
\[ p^{-1}(U)\cong \M|_U\times\F
\] for a fixed graded manifold $\F$ called the fiber of $p$.

A vector bundle $p:\E\to\M$ is a fiber bundle where the fiber is a graded vector space and the transition functions are chosen to be isomorphisms of graded vector spaces.

Given a vector bundle $\E\to\M$ we denote by $\E[k]$ the vector bundle with same fiber, but the grades of the fiber coordinates shifted by $k$.
\end{vdef}
Note that by $\rk\E$ we denote the dimension of the fiber of a vector bundle.

\begin{vdef} A \emph{submanifold} of a graded manifold $(M,\gsmooth)$ is a graded manifold $(U,\P)$ together with an injective map $i\:(U,\P)\to(M,\gsmooth)$.

The \emph{open} submanifolds are the submanifolds $(U,\P)$ where $U\subset M$ is open and $\P=\gsmooth|_U$.
\end{vdef}
\begin{vdef} Given a fiber bundle $p\:\E\to\M$ the \emph{sheaf of sections} $\Gamma(\bullet,\E)$ is the sheaf of maps $\{(f\:\U\to\E) : p\circ f=\id_\U\}$.
\end{vdef}

The next aim is to define the tangent bundle.  As for smooth manifolds we want to define it as the vector bundle underlying the derivations of the structure sheaf.  We need the following lemma.
\begin{lemma} A sheaf of modules $\P$ over a ringed space $(M,\gsmooth)$ is projective iff there is a vector space $F$ such that for every point $p\in M$ there is an open neighborhood $p\in U\subset M$ with
\[ \P(U) \cong \ssmooth(U)\otimes F \;.
\]
\end{lemma}

Therefore we have the following characterization of vector bundles.
\begin{prop}  Finite dimensional projective modules $\P$ over $\M$ are in one-to-one correspondence with finite dimensional vector bundles. \qed
\end{prop}

\begin{vdef} The \emph{vector fields} of a smooth manifold $\M=(M,\gsmooth)$ are the sheaf $\X$ of (graded) derivations of the structure sheaf $\gsmooth$.  
\end{vdef}
\begin{prop} The (sheaf of) vector fields is a (sheaf of) projective module(s). The \emph{tangent bundle} $T\M$ is the underlying vector bundle. It has rank equal to the dimension of the manifold, i.e.\ if $\dim\M=(\ldots,p_{-1},p_0,p_1,\ldots)$ then $\rk T\M=(\ldots,p_{-1},p_0,p_1,\ldots)$.
\end{prop}
\begin{proof}  Given a coordinate neighborhood $\U=(U,\gsmooth)$ of $\M$ with
 $$ \gsmooth(U)\cong\smooth(U)\otimes \Lambda^\bullet\R^{p_1}\otimes S^\bullet \R^{p_2}\otimes \ldots \R^{p_n} \;.
 $$  Then the sheaf of derivations restricted to $U$ is isomorphic to the sheaf
generated from $T\U:=\U\times(\R^{p_0}\oplus \R^{p_1}[1]\oplus\R^{p_2}[2]\oplus\ldots)$.
\end{proof}
Given a coordinate neighborhood $\U=(U,\gsmooth)$ of $\M$ with coordinates $x^i$ of degree 0, $\xi^a$ of degree 1, $b^B$ of degree 2, \ldots, we introduce the coordinate vector fields $\partial/\partial x^i$, $\partial/\partial\xi^a$ deriving from the left, $\partial/\partial b^B$, \ldots with the usual meaning and grading opposite to the degree of the coordinates.  Note that the fiber coordinates have opposite degree of the coordinate vector fields.

\begin{Exam}  Given a graded manifold $\M$ we can define the so-called Euler vector field $\epsilon\in\X(\M)$.  In a local chart this reads as
\begin{align*}  \oldepsilon[x^i] &= 0, \\
  \oldepsilon[\xi^a] &= \xi^a, \\
  \oldepsilon[b^B] &= 2b^B, \\
  \vdots \\
  \oldepsilon[q^\alpha] &= |\alpha|q^\alpha \;, \\
  \text{i.e.}\quad \oldepsilon &= \xi^a\pfrac{}{\xi^a} +2b^B\pfrac{}{b^B}+\ldots
\end{align*} where $q^\alpha$ is the compact notation for all coordinates and the degree of the (homogeneous) coordinate $q^\alpha$ is $|\alpha|$.  It is easy to check that $\oldepsilon$ is invariant under coordinate changes.
\end{Exam}

\begin{vdef}  Let $\M$ be a graded manifold.  The \emph{cotangent bundle} $T^*\M$ is the dual vector bundle of the tangent bundle.  Let $\Omega^1$ be the sheaf of $T^*\M$, and $\Omega^k=\Lambda^k\Omega^1$ is the $k$-th graded skew-symmetric power of $\Omega^1$.
\end{vdef}
Note that the graded skew-symmetric product is degree-wise isomorphic to the graded symmetric product with degrees shifted by 1, via
 \[\dec_k(\ud q^{\alpha_1}\wedge\ldots\wedge\ud q^{\alpha_k}) = (-1)^{\sum_{i=0}^{\lfloor(k-1)/2\rfloor}|x_{k-2i-1}|} sq^{\alpha_1}\cdot\ldots\cdot sq^{\alpha_k} \;.\]
Here $s\:V\xto{\sim}V[1]$.  However, we will use the notation as graded skew-symmetric tensor product.  This notion means that the sheaf $\Omega^\bullet$ has a double grading, i.e.\ a homogeneous form, e.g.\ $\omega=\omega_{aB}\ud\xi^a\wedge\ud b^B$, has a form degree, here 2, as well as a degree coming from the involved coordinates, here 3.  The sign when interchanging forms is governed by $\ud q^\beta\wedge \ud q^\alpha=(-1)^{|\alpha|\cdot|\beta|+1\cdot1}\ud q^\alpha\wedge \ud q^\beta$.
  We will denote the degree coming from the involved coordinates simply as degree.

\begin{prop} The action of vector fields factors through a map $\ud\:\gsmooth\to\Omega^1$, as 
\[ X[f]=(-1)^{|f|\cdot|X|}\<\ud f,X\>\;.\]  $\ud\:\gsmooth\to\Omega^1$ is a derivation
\[ \ud(fg)=(-1)^{|f|\,|g|}g\,\ud\! f +f\ud g \] called the de Rham differential.  It naturally extends to $\ud\:\Omega^\bullet\to\Omega^{\bullet+1}$.
The cotangent bundle has the rank opposite to the dimension of the manifold $\M$, i.e.\ for $\dim \M=(\ldots,p_{-1},p_0,p_1,\ldots)$ we have 
 $\rk T^*\M=(\ldots,p_1,p_0,p_{-1},\ldots)$.
\end{prop}
\begin{proof}  Analog to the coordinate vector fields of the tangent bundle $T\M$ we have coordinate one-forms $\ud x^i$, $\ud\xi^a$, $\ud b^B$,\ldots  over a coordinate chart of $\M$.  

If we write the vector field $X$ locally as
\begin{align*}
  X&=X^\alpha\pfrac{}{q^\alpha}
\intertext{then we can rewrite its application to a function $f\in\gsmooth(\M)$ as}
 X[f] &= X^\alpha\pfrac{f}{q^\alpha}= (-1)^{(|X|+|\alpha|)|f|}\left\<\pfrac{f}{q^\alpha}\cdot\ud q^\alpha,X^\beta\pfrac{}{q^\beta}\right\>
\intertext{with the convention}
  \left\<\ud q^\alpha,\pfrac{}{q^\beta}\right\> &= \delta^\alpha_\beta \;.
\intertext{We define the de Rham differential $\ud$ as}
  \ud\:\Omega^0\to\Omega^1: f\mapsto (-1)^{|f|\,|\alpha|}\pfrac{f}{q^\alpha}\cdot \ud q^\alpha\;.
\intertext{and for higher form degrees as}
 \ud(\omega_{\alpha_1,\ldots,\alpha_k}\cdot\ud q^{\alpha_1}\wedge\ldots\wedge\ud q^{\alpha_k}) 
  &= (-1)^{|\omega_{\alpha_1,\ldots,\alpha_k}|\,|\alpha_0|}\pfrac{\omega_{\alpha_1,\ldots,\alpha_k}}{q^{\alpha_0}}\cdot\ud q^{\alpha_0}\wedge\ldots\wedge\ud q^{\alpha_k}
\end{align*}  Given that $\ud$ is an odd derivation with respect to the form-degrees, it is a straightforward computation, analog to the ungraded case, that the map $\ud$ squares to 0.
\end{proof}

\begin{vdef}  A \emph{Q-structure} on a graded manifold $\M$ is a vector field $Q$ of degree +1 that is nilpotent, i.e.
 $$ [Q,Q]=0 \;. $$
\end{vdef}

An NQ-manifold is an $\NN$-manifold with a Q-structure.

\subsection{Graded symplectic and Poisson manifolds}
\begin{vdef} A \emph{graded symplectic} manifold is a graded manifold $\M$ together with a homogeneous two-form $\omega\in\Omega^2(\M)$ of degree $d$ which is closed under the de Rham differential and the induced map $\omega^\#:T\M\to T^*[d]\M$ is nondegenerate.

\index{graded manifold!Poisson}\index{Poisson manifold (graded)} A \emph{graded Poisson} manifold is a graded manifold $\M$ together with a graded Poisson bracket $\{.,.\}$ on the structure sheaf.  Let $\{.,.\}$ be of degree $d$. The graded Leibniz rule reads as
\begin{align}
  \{f,gh\} &= \{f,g\}h +(-1)^{(|f|+d)|g|}g\{f,h\}
\intertext{The graded Jacobi identity reads as}
  \{f,\{g,h\}\} &= \{\{f,g\},h\} +(-1)^{(|f|+d)(|g|+d)}\{g,\{f,h\}\}
\end{align}  where $|f|$ denotes the degree of the homogeneous function $f$.
\end{vdef}

Note that there is a graded Lie bracket on the vector fields.  It induces the Schouten bracket $[.,.]$ of degree 0 on multivector fields $\X^\bullet:=\Gamma(\bullet,\Lambda^\bullet T\M)$.

\begin{prop} A Poisson structure on a graded manifold $\M$ is uniquely determined by a bivector field $\Pi\in\Gamma(\M,\Lambda^2T\M)$ satisfying $[\Pi,\Pi]=0$. The relation is
\[ \{f,g\} = \<\ud f\wedge\ud g,\Pi\> \;.
\]

A symplectic structure $\omega$ of degree $-d$ induces a nondegenerate Poisson structure of degree $d$ as
 \[ \Pi^\#=(\omega^\#)^{-1} \;. \] \qed
\end{prop}
The proof is analog to the proof for the ungraded case.

\begin{Exam}  Given a pseudo Euclidean vector bundle $(E,(.,.))$ we can construct the graded Poisson manifold $E[1]$ with Poisson bracket of degree $-2$ as follows:
\begin{align*}
  \{\xi,\eta\} &= (\xi,\eta)\quad\forall \xi,\eta\in\Gamma(E)\cong \gsmooth(E[1])^{[1]} \\
  \{\xi,\eta_1\eta_2\} &= \{\xi,\eta_1\}\eta_2 +(-1)^{|\xi|\cdot|\eta_1|}\eta_1\{\xi,\eta_2\} \quad\forall\xi,\eta_i\in\gsmooth(E[1]) \\
  \{\eta,\xi\} &= (-1)^{|\eta|\cdot|\xi|} \{\xi,\eta\} \quad\forall \xi,\eta\in\gsmooth(E[1])
\end{align*}
\end{Exam}

\begin{prop} The cotangent bundle of a graded manifold is a symplectic manifold with symplectic potential, the Liouville form. \qed
\end{prop}


A \emph{symplectic realization} of a Poisson manifold $M$ is a Poisson morphism $\pi:S\to M$ from a symplectic manifold $S$. Minimality is considered with respect to the dimension of the fibers.

Recall that $T^*E$, the cotangent bundle over a vector bundle, is canonically isomorphic to $T^*E^*$, the cotangent bundle over the dual vector bundle, via Legendre transformation, see e.g.\ \cite{Mack94,Tul77}.  The two projections $T^*E\to E$ and $T^*E\cong T^*E^*\to E^*$ add up to a map $p:T^*E\to E\oplus E^*$.  Therefore we have
\begin{prop}  $T^*[2]E[1]$ is a minimal symplectic realization of the odd Poisson manifold $(E\oplus E^*)[1]$.
\end{prop}
\begin{proof}  It remains to check minimality, but this is obvious from the local trivialization $T^*E=T^*M\oplus E\oplus E^*$ where the first factor is the minimal symplectic realization of the trivial Poisson manifold $M$.
\end{proof}

\begin{prop}[{\v S}evera, Roytenberg]\label{p:sympReal}  Let $(E\to M,(.,.))$ be a pseudo Euclidean vector bundle. A minimal symplectic realization is
 \[ \E:= E[1]\times_{(E\oplus E^*)[1]} T^*[2]E[1] = i^*T^*[2]E[1]
 \] where $i\:E\to E\oplus E^*:\psi\mapsto \psi\oplus\frac12(\psi,.)$.
\end{prop}
The first notation describes $\E$ as a fiber product over $E\oplus E^*[1]$.  The second notation describes $\E$ as a pullback bundle over $i$.
 $$\begin{CD}  \E @>>> T^*[2]E[1] \\
 @VVV @V{p}VV \\
 E[1] @>{i}>> (E\oplus E^*)[1]
\end{CD}$$
\begin{proof} Consider the pullback bundle.  It is a presymplectic manifold.  By inspection it is symplectic where the symplectic form and the Poisson bivector field read as:
\begin{align*} \omega &= \ud p_i\wedge\ud x^i +\frac12g_{ab}\cdot\ud\xi^a \ud\xi^b \\
  \Pi &= \pfrac{}{p_i}\wedge\pfrac{}{x^i} +\frac12g^{ab}\cdot\pfrac{}{\xi^a}\,\pfrac{}{\xi^b}
\end{align*}  Here $x^i$ are the coordinates on the base $M$, $\xi^a$ are the fiber-linear coordinates of $E\to M$, which gain degree 1, and $p_i$ are the symplectic dual to the $x^i$ and have degree 2.  Finally $g_{ab}$ are the coefficients of the metric $(.,.)$ in the orthonormal base $\xi_a$ (dual to the $\xi^a$) and $g^{ab}$ is the inverse matrix.
\end{proof}

\subsection{Derived bracket}\label{s:derived}
\index{Courant!Super language}

The Courant bracket can be seen as a derived bracket, as observed by {\v S}evera and Roytenberg--Weinstein \cite{Royt02}.



Note that in the case of a symplectic N-manifold of degree $-d$, the Q-structure is equivalent to a nilpotent Hamiltonian $H$ of degree $d+1$, $H=\frac1{d+1}\i_Q\i_\oldepsilon\omega$, where $\oldepsilon$ is the Euler vector field.

\begin{thm}[Roytenberg \cite{Royt02}]\label{p:dBrack} Symplectic NQ-manifolds of degree 2 are in one-to-one
  correspondence with Courant algebroids.

Every Courant algebroid determines a minimal symplectic realization and a cubic function $H$.

A symplectic NQ-manifold gives rise to a derived bracket defined by
\[ \dor{\phi}{\psi}=\{\{H,\phi\},\psi\}\;,\label{dBrack}\] for functions of 
degree 1.  This is a bracket for the sections of the vector bundle $E\to M$ generated by the functions of degree 1 over the 
smooth manifold $M$,  whose structure sheaf are the functions of degree 0.
The anchor map reads  $$\rho(\psi).=\{\{\psi,H\},.\}\;,$$
and the inner product is
 $$ \<\phi,\psi\>= \{\phi,\psi\}$$
\end{thm}
\begin{proof}  In one direction there is the straightforward computation that the derived bracket fulfills the axioms of a Courant algebroid.  The Jacobi identity is equivalent to the nil-potency of $Q$.  Note that $\{\psi,H\}=Q[\psi]$.

For the converse direction we need a construction of $H$. Locally we can trivialize $E$ and $\E$ over $M$.  We choose an orthonormal frame $\xi_a$ of $E$ for this.  Let further $x^i$ be coordinates of $M$ and $p_i$ their symplectic duals in $T^*M$, here having degree 2.  Now we can interpret the anchor map $\rho$ as a cubic function on $\E$ and the 
trilinear form $C(\xi_1,\xi_2,\xi_3):=\langle\dor{\xi_1}{\xi_2},\xi_3\rangle$ as another cubic function on $E$.  
 $$ H(x,\xi,p):= \rho^i_a(x)\xi^ap_i -\frac16C_{abc}(x)\xi^a\xi^b\xi^c
 $$
It is another straightforward computation that this hamiltonian gives the original bracket on $E$.  This computation also shows that components $\partial/\partial x^i$ and $\partial/\partial\xi^a$ of $Q$ are unique in order to produce all components of the operation $\dia$ and anchor map $\rho$.

It remains to check whether the local choices of $H$ can be glued together to a global $H$ and whether all components of $Q$ are uniquely determined.  The ambiguity of $H$ is change by a Casimir function of degree 3.  All Casimir functions form the 0th class of the Poisson cohomology of $\E$.  Note that $\E$ is symplectic and thus Poisson cohomology coincides with the de Rham cohomology.  Moreover Roytenberg showed the following property of de Rham cohomology of graded manifolds
\begin{lemma}[Roytenberg \cite{Royt02}]\label{l:Higher} The higher de Rham cohomologies, i.e.\ of functions and forms of degree $\ge1$, vanish.
\end{lemma}
Finally the uniqueness of the whole $Q$ follows, because any $Q$ that leaves the symplectic form invariant determines a unique Hamiltonian $H$ via $H=\frac13\i_Q\i_\oldepsilon\omega$ where $\oldepsilon$ is the Euler vector field that assigns the grades to the coordinates.
\end{proof}

\begin{Exam} For the generalized twisted Courant algebroid on $E=A\oplus 
A^*$ the cubic Hamiltonian is $H=H_d+H_{twist}$, where $H_d$ is the 
Hamiltonian lift of the Lie algebroid differential $\ud_A$ to $\E:=T^*[2]A[1]$ which can be seen as an odd vector field on $A[1]$, and $H_{twist}\in\Gamma(\Lambda^3A^*)$ any A-3-form. In order for $H$ to be nil-potent we need 
 $\ud_AH_{twist}=0$.
\end{Exam}

\begin{Exam} Starting with the tangent bundle and its associated de Rham 
differential one can construct the supermanifold $\E=T^*[2]T[1]M$ and lift the 
de Rham differential to a Hamiltonian $H_d$. This Hamiltonian is nilpotent, 
because the de Rham differential is. Moreover it produces the Dorfman bracket 
(see \eqref{Cbrack}) on $TM\oplus T^*M$ as a derived bracket. Analog to the 
general Lie algebroid case we can twist by a closed 3-form $H_{twist}$ that 
will only occur in the 1-form component of the Dorfman bracket. This proves 
the axioms for the Dorfman bracket in a short, but abstract way. 
\end{Exam}

Moreover the cubic Hamiltonian permits us to define Courant cohomology as follows.
\begin{vdef}[Roytenberg \cite{Royt02}]\label{def:Q}\index{Courant!cohomology}  Let 
 $(E\to M,\dia,\rho,\langle.,.\rangle)$ 
be a Courant algebroid.  Let $H$ be a cubic Hamiltonian that generates the
Courant bracket on a minimal symplectic realization $\E$.  Then the \emph{Courant
cohomology} is the cohomology of the (graded) functions with respect to the
nilpotent operator $Q:=\{H,.\}$.
\end{vdef}

The lowest cohomology groups have some common interpretations.  $H^0$ is the space of functions on the base $M$ that are constant along the leaves of the anchor foliation. $H^1$ is the space of sections of $E$ that act trivially as $\psi\dia$ on $E$ modulo the exact sections, i.e.\ images of smooth functions under the operator $\D=\rho^*\ud$.  $H^2$ is the space of linear vector fields on $E$ that preserve the Courant algebroid structure, modulo those vector fields that are generated by sections of $E$ as $\psi\dia$.  $H^3$ is the space of infinitesimal deformations of the Courant algebroid structure modulo the ones generated by functions of degree 2.  $H^4$ governs the obstructions of extending infinitesimal deformations to formal ones.  Note that compared to usual deformation theory the labels of the cohomology groups are shifted by two.

Roytenberg's Lemma \ref{l:Higher} shows that the deformations of the standard Dorfman bracket are only the {\v S}evera--Dorfman brackets, see Example \ref{ex:Sev}.

\subsection{Dirac structures II, superlanguage}
Given a subbundle $L\to N$ of $E$, one can construct a submanifold of $\E$
as follows:  $$\L:= N^*[2]L[1]$$ 
Note that conormal bundle $N^*L$ is a subbundle of the cotangent bundle $T^*E$
restricted to $L$.  We call the embedding $j:\L\to\E$.

\begin{prop}A subbundle $L\to N$ of $E$ is Lagrangian iff $\L$ is
  Lagrangian in $\E$.
\end{prop}
\begin{proof}  Being Lagrangian is clearly a local property.  Therefore given a
  point $(x,l)$ in $\L$ one can choose a trivialized neighborhood and write
  $\L$ as $L\oplus N^*N$.  Now the proposition goes back to the following
  lemma.
\begin{lemma} Suppose $S$ is a submanifold of $M$.  Then $N^*S$ is a
  Lagrangian submanifold of $T^*M$.
\end{lemma}
\end{proof}

One defines the map $\ell:\Gamma E \to \gsmooth(\E)$ as follows.  A
section of $E$ can be identified with a section of $E^*$ via the inner
product.  Further one identifies sections of $E^*$ with fiberwise linear
functions on $E$.  These can further be pulled back to a function on $\E$.  The super
language makes it an odd linear function.

\begin{lemma}  Let $L$ be a Lagrangian subbundle of $E$, $\L$ the
  corresponding Lagrangian submanifold of $\E$ embedded via $j$.  Then a
  section $\psi$ of $E$ is a section of $L$ iff $j^*\ell\psi=0 $, i.e.\ the
  Hamiltonian vector field is tangent to $\L$. \qed
\end{lemma}
Consider taking the Hamiltonian vector field as a(n anti)morphism of Lie
algebras, i.e.\ \(\left[X_{\ell\psi},Q\right]=\pm X_{\{\ell\psi,H\}}\).  So one 
can express the derived bracket as iterated commutators of Hamiltonian vector 
fields.
\[ X_{\{\{\ell\psi,H\},\ell\phi\}} = \left[\left[X_{\ell\psi},Q\right],
  X_{\ell\phi}\right]
\]

Therefore we have the following theorem.
\begin{thm}\index{Dirac structure}  A Lagrangian subbundle $L\to N$ of a 
  Courant algebroid $E$ is a Dirac structure iff the pullback of the generating
  function $H$ under the embedding $j$ vanishes.
\end{thm}
\begin{proof}\MG{to be completed}  Suppose the pullback of $H$ vanishes.  Then its
   Hamiltonian vector field is tangential to $\L$.  As the previous lemma
   states the Hamiltonian vector fields of two sections $\psi$ and $\phi$
   of $L$ are also tangential to $\L$.  Therefore the derived bracket has a
   Hamiltonian vector field tangential to $\L$.  Again by the last lemma one
   concludes that it comes from a section of $L$, therefore the bracket
   gives a section of $L$.  The differential $\D$ maps functions that vanish on $N$ to sections of $L$, because $\D$ is just $Q$ restricted to graded functions of degree 0.  Therefore the anchor maps $L$ to $TN$.
 
 Consider conversely that the pullback of $H$ is not vanishing.  The
 following modification of a lemma from symplectic geometry implies that there are two
 sections whose bracket does not close in $L$.
 \begin{lemma}  Let $L$, $E$ and $\E$ be as before and suppose $\chi$ is a
   super function on $\E$.  Suppose further that for all sections $\psi\in
   \Gamma(L)$ the bracket of the Hamiltonian vector fields of $\chi$ and
   $\ell\psi$ gives a vector field tangential to $\L$.  Then the Hamiltonian
   vector field of $\chi$ is tangential to $\L$.
 \end{lemma}
\end{proof}

\section{Spectral Sequence of a filtered complex \MG{90\%}}
For details about spectral sequences refer to the original work by 
Cartan and Eilenberg \cite{CEha} or to McCleary's \emph{``Users' Guide \ldots''} \cite{SSeq85}.  
We will give a brief summary and an example here.

\begin{vdef} A \emph{spectral sequence} is a collection of differential bigraded 
vector spaces $\{E^{p,q}_r,d_r\}_{p,q\ge0}$ where the differential maps 
 $d_r:E^{p,q}_r\to E^{p+r,q+1-r}_r$, i.e.\ has bidegree $(r,1-r)$, $d_r\circ d_r=0$
and the next sheet $E^{\bullet,\bullet}_{r+1}$ is the cohomology of the previous
one.  It is understood that the spaces $E^{p,q}_r$ vanish for $p<0$ or $q<0$.
\end{vdef}
The picture of one sheet of a spectral sequence is given in figure \ref{f:sseq}, where the arrows represent the differential $d_r$ and each dot is a vector space $E^{p,q}_r$.  Note that in the lower left corner the differentials vanish.
\begin{figure}[h!]\begin{center}
\includegraphics[height=5cm,clip=true]{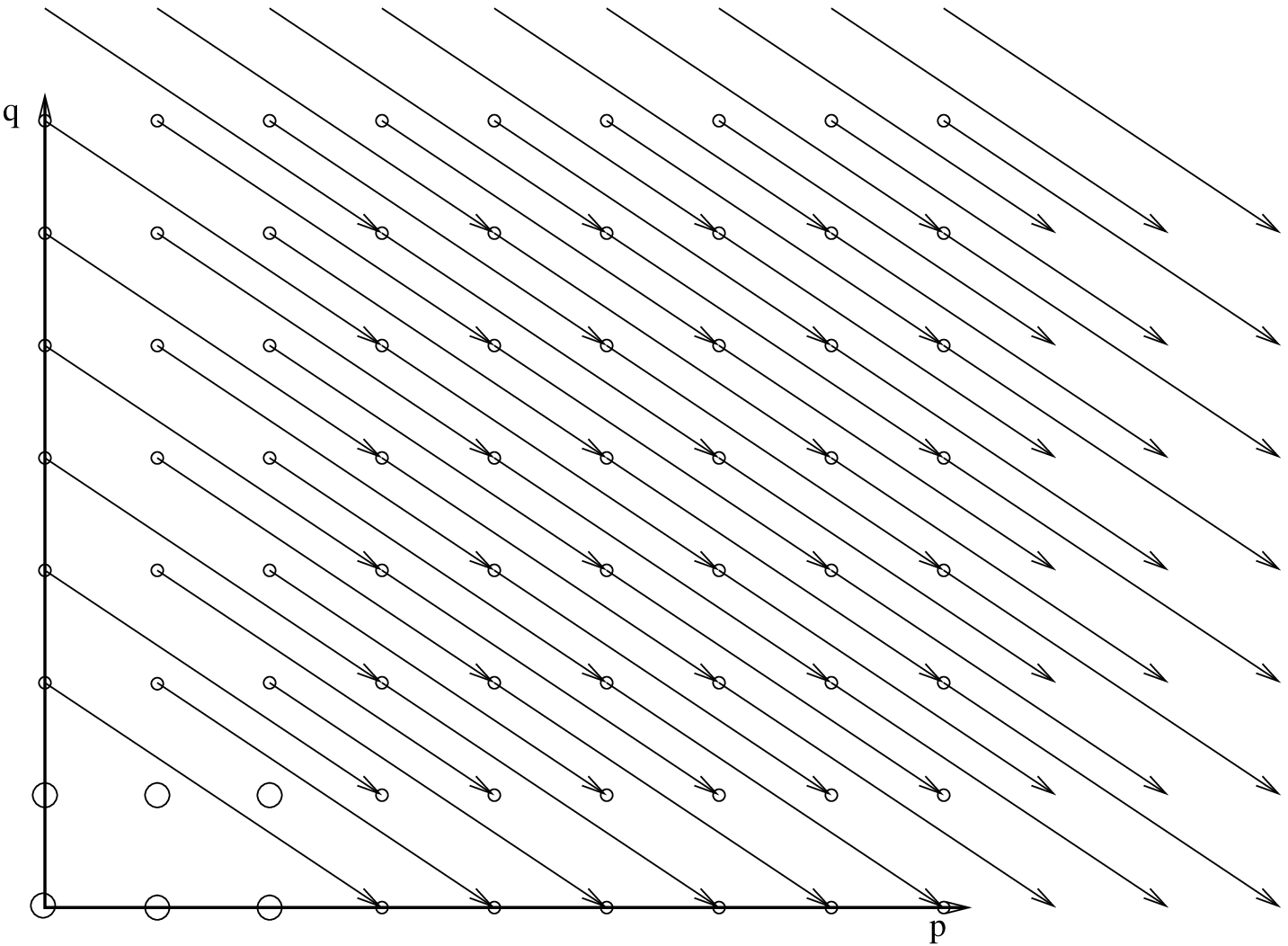}
\caption{\label{f:sseq}$\{E^{\bullet,\bullet}_3,d_3\}$ of a spectral sequence}
\end{center}
\end{figure}

Given a first sheet, e.g.\ $\{E^{\bullet,\bullet}_0\}$ and all differentials 
 $d_{r\ge0}$ then one can compute all higher sheets and indeed has a convergence to 
 $E^{\bullet,\bullet}_\infty$, because the higher differentials vanish on
the spaces in the lower left corner (see also the figure).\par\medskip

The question to ask next is how spectral sequences arise.  A broad family of 
spectral sequences occurs from filtrations of a differential complex.
\begin{vdef}  A (decreasing) \emph{filtered complex} $(F^pA^\bullet,d)$ is a 
  family of graded vector subspaces $F^{p+1}A^\bullet\subset F^pA^\bullet\ldots
   \subset F^0A^\bullet=A^\bullet$, such that the differential maps 
  $d(F^pA^n)\subset F^pA^{n+1}$.
\end{vdef}  The aim in this filtered complex is to compute the cohomology of the
complex which is often very complicated. It can be made easier when one has
such a filtration and descends to the associated graded (vector space):
 $$ E^{p,q}_0 := F^pA^{p+q}/F^{p+1}A^{p+q} \;.$$  
The associated graded has induced a simpler differential 
 $d_0:F^pA^{p+q}/F^{p+1}A^{p+q}\to F^pA^{p+q+1}/F^{p+1}A^{p+q+1}$, i.e.\ all 
the terms of $d$ in $F^{p+1}A^\bullet$ are suppressed.

The claim is now that this continues to all sheets and the spectral sequence 
computes the cohomology of the (unfiltered) complex:

\begin{thm}\label{thm:sseq}  A filtered differential graded module 
 $(A^\bullet,d,F^p)$ determines a spectral sequence (starting as written above).

Suppose further that the filtration is (totally) bounded, i.e.\ for each degree 
 $n$ there is an upper value $t=t(n)$ such that $F^pA^n=\{0\}$ for all $p\ge t(n)$.
Then the spectral sequence converges to H(A,d), namely
 $$ E^{p,q}_\infty\cong F^pH^{p+q}(A,d)/F^{p+1}H^{p+q}(A,d) \;.$$
\end{thm}  The proof can be found, e.g.\ in \cite[chap.\ 2.2]{SSeq85}

To see the theorem in a smaller example before applying it to the complex of a 
Courant algebroid let us compute the de Rham cohomology of the torus 
 $\T^2=\S^1\times\S^1$. One way to compute this cohomology is to use the 
K\"unneth formula. We will go a more general way which also applies to 
computing the cohomology in chapter \ref{s:regcase}. The cochains are
 $$A^n=\Omega^n(\T^2)\cong\begin{cases}
  \smooth(\S^1)\otimes\smooth(\S^1) \quad\quad \text{for }n=0\\
  \Omega^1(S^1)\otimes\smooth(\S^1)\oplus\smooth(\S^1)\otimes\Omega^1(\S^1)  \quad n=1 \\
  \Omega^1(\S^1)\otimes\Omega^1(\S^1) \quad\quad n=2  \\
  0 \quad\quad n>2
  
\end{cases}$$ i.e.\ just the tensor product of the forms on each factor 
 $\S^1$. Also the de Rham differential is the sum of the two de Rham 
differentials of each factor $\ud=\ud_A\otimes1+1\otimes\ud_B$. Note that in the 
tensor product notation as above with the odd 1-form $\alpha$, $(\ud_A+\ud_B)
(\alpha\otimes g)=(-1)\alpha\otimes \ud_B g$, because we changed the odd 
operator $\ud_B$ with $\alpha$. Let us observe that $I:=\Omega^\bullet(\S^1)
\otimes\Omega^{\ge1}(\S^1)$ is an ideal in the algebra $A^\bullet$ (the 
multiplication being the wedge product). Moreover this ideal is preserved by 
the differential, because $\ud(\alpha\otimes\beta)=(\ud_A\alpha)\otimes\beta$ 
for a form $\alpha$ of arbitrary degree and a 1-form $\beta$. Such an ideal 
implies a filtration, namely $F^pA^\bullet=I^p$ with $I^0:=A$ and here $I^2=0$,
 thus the situation is exactly as in the theorem $\{0\}=F^2A\subset 
F^1A\subset F^0A=A$, a bounded decreasing filtration.

The 0th sheet of the spectral sequence reads as follows:
 $$ E^{p,q}_0:=F^pA^{p+q}/F^{p+1}A^{p+q}\cong \begin{cases}
  \Omega^q(\S^1)\otimes\smooth(\S^1)+I \quad\text{for } p=0\\
  \Omega^{q}(\S^1)\otimes\Omega^1(\S^1) \quad p=1\\
  0 \quad\quad p\ge2
 \end{cases}
 $$  The differential is induced as explained earlier by
 $d_0=(\ud\:F^pA^{p+q}/F^{p+1}\to F^pA^{p+q+1}/F^{p+1})$.  The remaining part is 
 $d_0=\ud_A$.

Therefore the first sheet computes as:
 $$ E^{p,q}_1:=H^{p,q}(E_0,d_0)\cong\begin{cases}
  \smooth(\S^1)  \quad\text{for } q=0\text{ or }1, p=0\\
  \Omega^1(\S^1) \quad q=0\text{ or 1}, p=1\\
  0        \quad \text{otherwise}
\end{cases}$$  
The proof of theorem \ref{thm:sseq} also shows that there exists a unique
differential $d_1$ constructively.  We know that $d_1:E^{p,q}_1\to E^{p+1,q}_1$.
Moreover the only part remaining from the original $\ud$ is $\ud_B$ which maps 
exactly this way.  So it is not surprising that $d_1=\ud_B$.

Thus the second sheet reads as
 $$E^{p,q}_2:=H^{p,q}(E_1,d_1)\cong\begin{cases}
  \R \quad\text{for } q=0\text{ or }1, p=0 \\
  \R \quad  q=0\text{ or }1, p=1\\
  0  \quad\text{otherwise}
\end{cases}$$  At this step the new differential maps 
 $d_2:E^{p,q}_2\to E^{p+2,q-1}_2$, but either source or target space are trivial, 
so $d_2=0$ and also all the higher differentials.

Therefore the spectral sequence converges to 
 $E^{\bullet,\bullet}_\infty=E^{\bullet,\bullet}_2$.  In terms of vector spaces
the cohomology reads now:
 $$H^n(\T^2,\ud)\cong\bigoplus_{n=p+q}E^{p,q}_\infty=\begin{cases}
  \R \quad\text{for }n=0\text{ or }2  \\
  \R\oplus\R \quad n=1  \\
  0 \quad\text{otherwise}
\end{cases}$$

Some remarks about the underlying algebraic structure. Clearly $\ud$ is not 
just a differential of vector spaces, but a differential of the algebra 
 $\Omega^\bullet(\T^2)$. Since the ideal also preserves the product (is an 
 ideal rather than a mere normal subgroup), all sheets of the spectral sequence 
are also bigraded differential algebras. This can simplify the computation of 
the next sheets. However, the algebraic structure of the cohomology can be 
reconstructed when $E^{\bullet,\bullet}_\infty$ is a 
free bigraded commutative algebra. Thus in our case we obtain $H^\bullet(\T^2)
\cong \R\{\xi_1,\xi_2\}$ the graded commutative algebra in two odd generators 
(in accordance with the K\"unneth formula).

\chapter[Computation of the Cohomology of Courant \ldots]{Computation of the Cohomology
 of Courant algebroids with split base \MG{80\%}}\label{s:comput}
In this chapter we 
compute the spectral sequence in the case of Courant algebroids with split 
base \cite{GG08}. In particular we prove the Conjecture of Sti\'enon--Xu (see Corollary~\ref{c:trans}).

Remember that a Courant algebroid has a generating cubic Hamiltonian $H$ on the
symplectic realization $\E$ of the underlying pseudo Euclidean vector bundle.
This implies a differential $\{H,.\}=Q$, of degree 1, and acting on functions
on $\E$.  Therefore it is possible to define
\begin{vdef}\label{def:Q}\index{Courant!cohomology}  Let 
 $(E\to M,\dia,\rho,\langle.,.\rangle)$ 
be a Courant algebroid.  Let $H$ be a cubic Hamiltonian that generates the
Courant bracket on a minimal symplectic realization $\E$.  Then the \emph{Courant
cohomology} is the cohomology of the (graded) functions with respect to the
nilpotent operator $Q:=\{H,.\}$.
\end{vdef}

The price for the generalization of Lie algebroid cohomology is that one has
to deal with graded manifolds instead of an exterior algebra.

\section{Naive cohomology}\label{s:naive}
In this section, we recall the definition of the naive cohomology of a
Courant algebroid~\cite{SX07}. It is  less involved than
Definition~\ref{def:Q}; for instance, it does not use a symplectic realization
of $E$.

Mimicking the definition of the differential giving rise to the cohomology of Lie   algebroids, the idea is to consider an operator $\ud: \Gamma(\Lambda^\bullet E)\to \Gamma(\Lambda^{\bullet +1} E)$ given    by
the Cartan formula\,:
\begin{align} \label{eq:d} \langle\ud\alpha, \psi_1\wedge\ldots\psi_{n+1}\rangle :=
  &\sum_{i=1}^{n+1} (-1)^{i+1}\rho(\psi_i)\langle\alpha,\psi_1\wedge\ldots
    \widehat{\psi_i}\ldots\wedge\psi_{n+1}\rangle  \\
  &+\sum_{i<j} (-1)^{i+j} \langle\alpha,\psi_i\dia\psi_j\wedge\psi_1\wedge
    \ldots\widehat{\psi_i}\ldots\widehat{\psi_j}\ldots\psi_{n+1}\rangle  \nonumber
\end{align}
where $\psi_1,\dots \psi_{n+1}$ are sections of $E$ and $\alpha\in 
\Gamma(\Lambda^n E)$ is identified with an $n$-form on $E$, \emph{i.e.}, a 
section of $\Lambda^n E^*$ by the pseudo-metric. However, the 
formula~\eqref{eq:d} is not well defined because it is not $\smooth(M)$-linear 
in the $\psi_i$ (due, for instance, to axioms~\ref{ax:Leibn} and 
\ref{ax:nSkew} in Definition~\ref{d:Courant}.
 The operator $\ud$ does not square to zero either since it is not skew-symmetric
in the $\psi_i$.\footnote{Using the skew-symmetric bracket does not help either, 
because it only fulfills a modified Jacobi identity.}  Nevertheless, 
Sti{\'e}non--Xu~\cite{SX07} noticed that
the formula~\eqref{eq:d} for $\ud$ becomes $\smooth(M)$-linear in the
 $\psi_i$ when one restricts to $\alpha\in\Gamma(\Lambda^n(\ker\rho))$.\footnote{
In this case using the skew-symmetric or Jacobi-fulfilling bracket does not matter
since the difference is exact, thus vanishes in the inner product with $\ker\rho$.} 
In fact they proved the following.

\begin{lemma}\label{lem:d}
 Formula~\eqref{eq:d} yields a well defined operator $\ud:\Gamma(\Lambda^\bullet \ker\rho)\to \Gamma(\Lambda^{\bullet +1} \ker\rho)$. Moreover, one has $\ud \circ \ud =0$.
\end{lemma}
Note that  $\ker \rho$ may be a singular vector bundle.\footnote{In this
 case $\Gamma(\ker \rho)$ means smooth
 sections of $E$ that are pointwise in the kernel of $\rho$.  We define
 similarly sections of  $\Lambda^\bullet\ker\rho$ where  $\rho$ has been
 extended as an odd $\smooth(M)$-linear derivation $\Lambda^\bullet E\to
 \Lambda^\bullet E\otimes TM$.}

\begin{proof} 
The first claim follows 
from the fact that the failure of the Leibniz rule in the left-hand side
is an exact term, \emph{i.e.}, is in the image of $\rho^*$ and that $\rho
 \circ \rho^*=0$.  Now the terms for $\ud\circ\ud$ add up to zero using
the Jacobi identity as in the Lie algebroid case, since all terms are equivalent
to those using the skew-symmetric bracket. See~\cite{SX07} Section 1 for more details.
\end{proof}

\begin{rem}\label{rem:deRham}
For a Lie algebroid $A$, the space $(\Gamma(\Lambda^\bullet A^*),\ud)$, where
 $\ud$ is the operator given by formula~\eqref{eq:d}, defines its cohomology.  
In particular, it calculates the de Rham cohomology of $M$ when $A=TM$.
\end{rem}

Thanks to the pseudo-metric $\langle.,.\rangle $ on $E$, one
can view $\Gamma(\Lambda^n E) $ as graded functions on $E[1]$ (of degree n),
which can further be pulled back to the minimal symplectic
realization $\pi: \E\to E[1]$.  Thus, we can identify
 $\Gamma(\Lambda^n(\ker\rho))$ with a subalgebra of $\smooth(\E)$.  Sti\'enon--Xu \cite{SX07} proved the following Proposition.
\begin{prop}\label{lem:Q=d}
 The Q-structure $Q=\{H,.\}$ (see Definition~\ref{def:Q}) maps
 $\Gamma(\Lambda^\bullet(\ker\rho))$ to itself. Moreover, if $\alpha \in \Gamma(\Lambda^\bullet(\ker\rho))$, then $Q(\alpha)=\ud(\alpha)$.
\end{prop}
In other words, $Q$ restricted to $\Gamma(\Lambda^\bullet(\ker\rho))$ coincides
with the differential $\ud$ given by formula~\eqref{eq:d} (note that
Proposition~\ref{lem:Q=d} also implies that $\ud\circ \ud=0$).
Since $\ud$ squares to 0,  Sti{\'e}non--Xu defined\,:

\begin{vdef}\label{d:naive}\index{Cohomology!naive}  Let $(E,[.,.],\rho,\langle.,.\rangle)$ be a
Courant algebroid.  The \emph{naive cohomology} of $E$ is the cohomology of  the sections of
 $\Lambda^\bullet\ker\rho$ equipped with the differential $\ud$ given by  the Cartan-formula~\eqref{eq:d}.
\end{vdef}
We denote $H_{naive}^\bullet(E)=H^\bullet(\Gamma(\Lambda^\bullet\ker\rho),\ud)$
the naive cohomology groups of $E$.  By Proposition~\ref{lem:Q=d}, there is a
canonical morphism $\phi:H_{naive}^\bullet(E)\to H^\bullet_{std}(E) $ from the
naive cohomology to the  (standard) cohomology of Courant algebroids, see~\cite{SX07}. We will
prove that this morphism is an isomorphism in the transitive case, see
Corollary~\ref{c:trans}.

\section[Geometric spectral sequence for cohomology of Courant \ldots]{Geometric 
 spectral sequence for cohomology of Courant algebroids with split base}
 \label{s:regcase}
In this section we define a spectral sequence converging to the cohomology of a
Courant algebroid.  Then, in the case of Courant algebroids with split base,
we compute the spectral sequence in terms of \notneeded{smooth} geometrical
data.  For details about spectral sequences, refer to \cite{CEha, SSeq85}.

\subsection{The naive ideal spectral sequence}\label{s:naiveSS}
The algebra $A^\bullet:=C^\infty(\E)$ of graded functions on $\E$ is endowed
with a natural filtration induced by the ideal generated by the kernel of the
anchor map $\rho: E\to TM$.   More precisely, let $I$ be the ideal
 $$I:=\Gamma(\Lambda^{\ge1}\ker\rho)\cdot C^\infty(\E)\;,$$
{\it i.e.},  the ideal of functions containing at least one coordinate of
 $\ker\rho$, where we identify sections of $E$ to odd functions on $E$ by the
pseudo-metric (as in Section~\ref{s:naive}). Since $\ker \rho$ gives rise to
the naive cohomology, we call $I$ the naive ideal of $\E$.
\begin{lemma} $I$ is a differential ideal of the differential graded algebra
 $(A^\bullet,Q)$.
\end{lemma}
\begin{proof}
According to Proposition~\ref{lem:Q=d} and Lemma~\ref{lem:d}, we have  $Q(\Gamma(\Lambda^n \ker\rho)) \subset\Gamma(\Lambda^{n+1} \ker \rho)$. The result follows since $Q$
is a derivation.
\end{proof}

Since $I$ is a differential ideal, we have a decreasing bounded (since $E$ is 
finite dimensional) filtration of differential graded algebras $A^\bullet= 
 F^0A^\bullet \supset F^1A^\bullet \supset F^2A^\bullet \dots$, where $F^pA^q:= 
 I^p\cap A^q$. Therefore\,:
\begin{prop}\label{l:SS}
There is a  spectral sequence of algebras
 $$E_0^{p,q}:=F^pA^{p+q}/F^{p+1}A^{p+q} \Longrightarrow H^{p+q}_{std}(E) $$
converging to the cohomology of the Courant algebroid  $E$.
\end{prop}
We call this spectral sequence the \emph{naive ideal spectral sequence}.
\begin{proof} The spectral sequence is the one induced by the filtration $F^\bullet A^\bullet$ of the complex $(A^\bullet,Q)$. It is convergent because the filtration is bounded.
\end{proof}
\begin{rem}
In order for the naive ideal spectral sequence to be useful, one needs to be
able to calculate the higher sheets $E_k^{p,q}$ of the spectral sequence in
terms of (smooth) geometry of $E$. Such calculations involve the image of the
anchor map. Thus it seems reasonable to restrict to the class of regular
Courant algebroids. \GGi{For general regular Courant algebroids,
the foliation groupoid given by the integral leaves of $D$ may be complicated
and suggests to make the computation in the framework of stacks, see
Remark~\ref{rem:stack} below.}\MG{It might not even be a stack, because the
groupoid is not Lie either.  See my notes after Remark~\ref{rem:stack}.}
\end{rem}

\subsection{Courant algebroids with split base} \label{s:split}

In this section we compute explicitly the naive ideal spectral sequence
of Proposition~\ref{l:SS} for what we call Courant algebroid with split base
and then prove the conjecture of Sti\'enon--Xu as a special case.

\medskip

By the bracket preserving property of the anchor-map $\rho$, $D:=\img\rho$ is 
an integrable (possibly singular) distribution.
\begin{vdef}\label{d:split}  A Courant algebroid $(E\to M,\langle.,.\rangle,
 [.,.],\rho)$ is said to have \emph{split base} iff $M\cong L\times N$ and the image of
 the anchor map is $D:=\im\rho\cong TL\times N\subset TM$.
\end{vdef}
In particular, a Courant algebroid with split base is a regular Courant
algebroid.  Furthermore, the integral leaves of the distribution $D$ are
smoothly parametrized by the points of $N$; thus the quotient $M/D$ by the
integral leaves of the foliation is isomorphic to $N$.

\begin{rem} \label{rem:lcoo}
The local coordinates $\xi^a$ for $E$ introduced in proof of Theorem \ref{p:dBrack} can 
be splitted accordingly to the isomorphism $D\cong TL\times N$. This splitting 
is useful in order to do local computations. More precisely, over a 
chart-neighborhood $U$ of $M$, $D:=\img\rho$ can be spanned by coordinate 
vector fields $\pfrac{}{x^I}$. Let $\xi^I:=\rho^*\ud x^I$ ($\in \ker \rho$) be 
vectors that span the image of $\rho^*$, and let $\xi_I$ be preimages of 
 $\partial_I$, dual to the $\xi^I$. Then choose coordinates $\xi^A\in \ker 
\rho$ normal (with respect to the pseudo-metric) and orthogonal to both the 
 $\xi^I$s and the $\xi_I$s. Therefore we have split the coordinates $\xi^a$s in 
the three subsets consisting of the $\xi^I$s, the $\xi_I$s and the $\xi^A$s. 
Furthermore this splitting also induces a splitting of the degree 2 
coordinates $p_i$s (the conjugates of the coordinates on $M$) into the $p_I$s, 
which are the symplectic duals of the $x^I$, and the $p_{I'}$s (their 
complements for which the Poisson bracket with the $x^I$s vanish). Since 
 $D\cong TL\times N$, the coordinates $x^I$ and $x^{I'}$ can be chosen to be 
coordinates of $L$ and $N$ respectively.

The Hamiltonian in these coordinates reads as $H=p_I\xi^I+\frac16C_{abc}(x)
\xi^a\xi^b\xi^c$.
Note that in order to compute structure functions like $C_{abc}$ you need
 $\xi_A$ which is the dual frame of $\xi^A$ or due to the pseudo orthonormality
 $\xi_A=\pm\xi^A$.
\end{rem}

\medskip

We denote $\X^q(N)$ the space of (degree q) symmetric multivector fields 
 $\Gamma_N(S^{q/2}(TN))$ with the convention that $S^{q/2}(TN)$ is $\{0\}$ for 
odd $q$'s, \emph{i.e.}, $\X^\bullet(N)$ is concentrated in even degrees. With 
these notations, the sheet $E_1^{\bullet,\bullet}$ of the spectral sequence of 
Lemma~\ref{l:SS} is given by\,:
\begin{lemma}\label{l:E1} For a Courant algebroid with split base $D\cong TL\times N$,
  one has
\begin{align*}
  E_1^{p,q}\cong \Gamma_M\Lambda^p(\ker\rho)\otimes \X^q(N).
\end{align*}
\end{lemma}

\begin{proof}
The differential $\ud_0:E_0^{p,q}\to E_0^{p,q+1}$  is the
differential induced on the associated graded $\bigoplus
 F^pA^{p+q}/F^{p+1}A^{p+q}$ of the naive filtration $F^\bullet
 A^\bullet$. Hence, it is obtained from $Q$ by neglecting all terms which
contain at least one term in $I$.  Using the local coordinates of
Remark~\ref{rem:lcoo}, we find $\ud_0 = p_I\pfrac{}{\xi_I}$.  Globally, the
algebra $E_0^{0,\bullet}$ is isomorphic to $\gsmooth(B)$ for the graded
manifold $B:=\E/I$ which, on a local chart $U\subset M$ is isomorphic to
 $B_{|U}\cong T^*[2]U\otimes_U D[1]|U$.  Globally, $B$ fits into the short exact
sequence of graded fiber bundles over $M$\,:
\begin{equation}\label{eq:Bseq}  0\to T^*[2]M\to B\to D[1]\to 0\;,
\end{equation} where the latter map
  $B\to D[1]$ is the map induced by the anchor map $\rho: E \to D=\im\rho$
on the quotient of $\E$ by $I$.
The differential $\ud_0$ on $E_0^{0,\bullet}\cong \smooth(B)$ is canonically 
identified with the odd vector field $\tilde\rho_0$ induced by $Q$ on the 
quotient $B=\E/I$.

There is a similar interpretation of $E_0^{1,\bullet}$. Precisely, $E_0^{1,
\bullet}$ is isomorphic to $\Gamma_B(B_1)$ for a (graded) vector bundle 
 $B^1\to B$ which, locally, is the fiber product $B^1_{|U}\cong 
\ker\rho[1]\times_U B$. In particular, $B^1$ fits into the short exact 
sequence of graded fiber bundles over $M$\,:
\begin{equation}\label{eq:B1seq} 0\to\ker\rho[1]\to B^1\to B\to 0
\end{equation}
The isomorphism $E_0^{1,\bullet}\cong \Gamma_B(B_1)$ identifies the 
differential $\ud_0:E_0^{1,\bullet}\to E_0^{1,\bullet+1}$ with the odd vector 
field $\tilde\rho: \Gamma_B(B^1)\to \Gamma_B(B^1)$ defined as the covariant 
derivative $\tilde\rho=\nabla_{\tilde\rho_0}$ along $\tilde{\rho}_0$ where $\nabla$ 
is a local connection on $B_1$ vanishing on a local frame $\xi^a$ of $B^1$. 
This is well defined, because the transition functions between such frames 
come from functions on $M$ and $\tilde{\rho_0}$ projected to $M$ vanishes (we 
extend to arbitrary sections of $B^1$ via Leibniz rule).

This identification of $E_0^{1,\bullet}$ extends to the other lines
 $E_0^{p\geq 2,\bullet}$ of the spectral sequence easily.  Namely, there is an
isomorphism $E_0^{p,\bullet}\cong \Gamma_B(S^pB^1)$, where $S^p$ stands for
the graded symmetric (hence it is skew-symmetric since the fibers
 $\ker\rho[1]$ are of odd degree) product over $B$. We extend $\tilde{\rho}$
to $ \Gamma_B(S^pB^1)$ by the Leibniz rule. Since, by Lemma~\ref{l:SS},
 $\ud_0$ is a derivation, we also have the identification of $\ud_0:E_0^{p\geq
  2,\bullet}\to E_0^{p\geq 2,\bullet+1}$ and $\tilde{\rho}$.

\smallskip

According to Lemma~\ref{l:SS}, $\ud_0$ is a derivation, thus it is sufficient 
to compute the cohomology of the complex $(E_0^{0,\bullet},d_0)$.\MGI{It
  seems we had to compute $E^0$ and $E^1$} The 
sequence~\eqref{eq:Bseq} yields a morphism of sheaves $\eta: \smooth(B) \to 
\Gamma_M(S^\bullet(TM)[2]) \to \Gamma_M(S^\bullet(TM/D)[2])$. On a local chart,
 the complex $(E_0^{0,\bullet},d_0)$ is isomorphic to the Koszul complex of 
 $\B^\bullet:=\Gamma_M(S^\bullet(TM)[2]\otimes\Lambda^\bullet(E/\ker\rho))$ 
with respect to the differential $\tilde\rho= p_I\pfrac{}{\xi_I}$ induced by 
the regular family given by the $p_I$s. Since the image of $\tilde \rho$ spans 
 $D\cdot \B^\bullet$, the morphism of sheaves $\eta:\smooth(B) \to 
\Gamma_M(S^\bullet(TM/D)[2])$ is locally a quasi-isomorphism. Thus we have 
\begin{align*}
  H^q(E_0^{0,\bullet},d_0)&\cong \Gamma_M(S^\bullet(TM/D)[2]) \\ 
  &\cong \Gamma_{N}(S^\bullet T(N)[2])\otimes_{\smooth(N)}\smooth(L\times N) \\ 
  &\cong \X^q(N) \otimes_{\smooth(N)}\smooth(M) \;,
\end{align*} where the second line follows from the fact that $E$ has split 
base $D\cong TL\times N$. 
The computation of the line $E_1^{1,\bullet}$ is similar since $\tilde\rho$ is a 
horizontal lift of $\tilde{\rho}_0$. In particular, $\tilde{\rho}$ does not 
act on the fibers $\ker\rho[1]$ of $B^1\to B$, but only on the base $B$. Hence 
 $$H^q(E_0^{1,\bullet},d_0)\cong \X^q(N) \otimes \Gamma_M(\ker\rho).$$ Since 
the lines, $E_0^{p,0}$ are generated by (products of elements of) $E_0^{1,0}$ 
and $\ud_0$ is a derivation, the result follows.
\end{proof}

\begin{rem} \label{rem:stack}\PX{Shorten this remark\,!}
Lemma~\ref{l:E1} is the main reason to restrict to Courant algebroids with
split base.  In general one can consider the quotient $M/D$  of $M$ by the 
integral leaves of the integrable distribution $D:=\img\rho$.  Let
 $\smooth(M)^D$ be the space of smooth functions on $M$ constant
along the leaves and $\X_{flat}(M/D)$  the space of derivations of 
 $\smooth(M)^D$.  To describe the spectral sequence using 
smooth geometry,  we would like a formula of the form\,: $$E_1^{p,q}\cong 
\Gamma_M(\Lambda^p \ker \rho)\otimes S^{q/2}(\X_{flat}(M/D)). $$ A quick analysis of 
the proof of Lemma~\ref{l:E1} shows that this formula will hold if and only if 
we have the relation
\begin{equation}\label{eq:badquotient}\Gamma_M(TM/D)
  \cong \Gamma_{M/D}(T_{flat}(M/D))\otimes_{\smooth(M)^D}
\smooth(M) \;.
\end{equation}
Courant algebroids with split base are a large class for which 
relation~\eqref{eq:badquotient} holds. However Relation~\eqref{eq:badquotient} 
does not hold for every regular Courant algebroid. For instance, take the Lie 
algebroid $D$ underlying the irrational torus. That is $M=\T^2$ is foliated by 
the action of a non-compact one parameter subgroup of $\T^2$ and $D$ is the 
subbundle of $TM$ associated to the foliation. The leaves are dense. Let 
 $E=D\oplus D^*$ be a generalized exact Courant algebroid (as in 
Example~\ref{e:genexactsplitbase}). Note that $D$ is regular of rank $1$, thus 
 $\Gamma_M(TM/D)$ is non zero but $\smooth(M)^D\cong \R$, thus $\X_{flat}(M/D)
 =0$. Thus formula~\eqref{eq:badquotient} does not hold for $E$.
\end{rem}

\bigskip

Recall from Section~\ref{s:naive} that the naive cohomology
 $H_{naive}^\bullet(E)$ is the cohomology of the complex
 $(\Gamma_M\Lambda^\bullet(\ker\rho),\ud)$.  The second sheet of the spectral sequence is
computed by $H_{naive}^\bullet(E)$\,:
\begin{prop} \label{p:E2} Let $E$ be a Courant algebroid with split base
  $D\cong TL\times N$. Then, one has an isomorphism of graded algebras
 $$E_2^{p,q} \cong H_{naive}^p(E)\otimes \X^q(N). $$
\end{prop}
\begin{proof}
It is a standard fact of spectral sequences~\cite{CEha, SSeq85} that the
differential
 $$\ud_1: E_1^{p,q}\cong H^q(F^pA/F^{p+1}A,\ud_0) \to E_1^{p+1,q}\cong
 H^q(F^{p+1}A/F^{p+2}A,\ud_0)$$
is the connecting homomorphism in the
cohomology long exact sequence induced by  the short exact sequence of
complexes $0\to F^{p+1}A/F^{p+2}A \to F^{p}A/F^{p+2}A \to F^{p}A/F^{p+1}A \to
0$.
 On a local chart, we obtain that $\ud_1$ is given by the following formula
\begin{align*}\ud_1
  =& \xi^I\pfrac{}{x^I}+C_{IA}^K\xi^I\xi^A\pfrac{}{\xi^K}
    +\frac12C_{AB}^K\xi^A\xi^B\pfrac{}{\xi^K}
    +\frac12C_{IJ}^A\xi^I\xi^J\pfrac{}{\xi^A} \\
  &+C_{AI}^B\xi^A\xi^I\pfrac{}{\xi^B}+\frac12C_{AB}^C\xi^A\xi^B\pfrac{}{\xi^C}
\end{align*}
where we use the local coordinates introduced in Remark~\ref{rem:lcoo}.
Now it follows from the isomorphism $E_1^{p,q}\cong
\Gamma_M\Lambda^p(\ker\rho)\otimes \X^q(N)$ given by Lemma~\ref{l:E1} and the
above formula for $\ud_1$ that  $$\ud_1=\ud\otimes 1:
\Gamma_M\Lambda^p(\ker\rho)\otimes \X^q(N) \to
\Gamma_M\Lambda^{p+1}(\ker\rho)\otimes \X^q(N)$$ where $\ud$ is the naive
differential. The result follows.
\end{proof}

The third sheet of the spectral sequence is trivially deduced from Proposition~\ref{p:E2}.
\begin{cor}\label{c:E3} Let $E$ be a Courant algebroid with split base. There
  is a canonical isomorphism  of bigraded algebras
 $E_3^{\bullet,\bullet} \cong E_2^{\bullet,\bullet}$.
\end{cor}
\begin{proof}
 Since $\X^q(N)$ is concentrated in even degrees for $q$, so is $E_2^{p,q}$
 by Proposition~\ref{p:E2}. Therefore, the differential $\ud_2:E_2^{p,q}\to
 E_2^{p+2,q-1}$ is necessarily $0$. Hence $E_3^{\bullet,\bullet} \cong
 E_2^{\bullet,\bullet}$.
\end{proof}

\bigskip

We  now prove the conjecture of Sti{\'e}non--Xu.  There is a canonical
morphism $\phi:H^\bullet_{naive}(E) \to H^\bullet_{std}(E)$, see~\cite{SX07}
and  Section~\ref{s:naive}.
\begin{cor}\label{c:trans} Let $E$ be a transitive Courant algebroid. Then
  the canonical map $\phi$ is an isomorphism $\phi:H^\bullet_{naive}(E)\cong 
 H^\bullet_{std}(E)$, {\em i.e.}, the Courant algebroid cohomology
coincides with the naive cohomology.
\end{cor}
\begin{proof}
A transitive Courant algebroid satisfies $D:=\img\rho=TM$.  Therefore, $E$ is
trivially with split base and
 $M/D=\pt=N$.  Hence $\X^q(N)$ is non zero only for $q=0$, where it is
 $\R$. Therefore, by Proposition~\ref{p:E2},  $E_2^{p,q}=0$ if $q\neq 0$. It
follows that all the higher differentials $\ud_{n\geq 3}$ are null. Thus the
cohomology of the Courant algebroid is isomorphic to $E_2^{\bullet,0}\cong
H^\bullet_{naive}(E)\otimes \R$.

\smallskip

Furthermore, by definition of the naive ideal $I$, the map $\phi$ preserves
the filtration by $I$. Thus $\phi$ passes to the spectral
sequence and coincides with the morphism of complexes
 $\phi_1:(\Gamma(\Lambda^n \ker \rho),\ud)\cong  (E_1^{n,0},\ud_1)
\hookrightarrow (\oplus_{p+q=n} E_1^{p,q},\ud_1)$  on the first sheet of the
spectral sequence. By the first paragraph of this proof,  $\phi_1$ is a
quasi-isomorphism. Hence $\phi$ is indeed an isomorphism.
\end{proof}

\begin{rem}\label{rem:conjectureexact}
When $E$ is an exact Courant algebroid,  Corollary~\ref{c:trans} was obtained
by Roytenberg~\cite{Royt02}. Indeed, it was one of the examples motivating the
conjecture of Sti\'enon--Xu.
\end{rem}
\bigskip

For general Courant algebroids $E$ with split base, the spectral sequence
does not collapse on the sheet $E_2^{\bullet,\bullet}$ but is
controlled by a map from the vector fields on $N$ to the naive
cohomology of $E$ which we now describe.

It is a general fact from the spectral sequence theory that there exists a 
differential $\ud_3:E_3^{p,q}\to E_3^{p+3,q-2}$. In particular, $\ud_3$ 
induces an $E_3^{0,0}=\smooth(N)$-linear map 
\begin{equation} \label{eq:T3} 
  T_3:\X(N) \to H^3_{naive}(E) 
\end{equation} given as the composition 
 $$ T_3:\X(N)\cong \X^2(N)\cong E_3^{0,2}\stackrel{\ud_3}\longrightarrow E_3^{3,
0}\cong H^3_{naive}(E) \;.$$ We call the map $T_3$ the \emph{transgression 
homomorphism} of the Courant algebroid $E$. Let $\X^{kil}(N)$ 
be the kernel of $T_3$ above (we 
like to think of elements of $\X^{kil}(N)$ as Killing vector fields 
preserving the structure function $H$). Note that $\X^{kil}(N)$ may be 
singular, {\em i.e.}, its rank could vary. We denote $\X^{kil, q}$ the 
space of ``symmetric Killing multivector fields'' $S^{q/2}_{\smooth(N)
}(\X^{kil}(N))$ with the convention that $\X^{kil, q}=\{0\}$ for odd 
 $q$.

\begin{thm}\label{thm:split} The  cohomology of a Courant algebroid $E$ with
  split base is given by
\begin{equation}\label{eq:split}  H^n_{std}(E)
  \cong \bigoplus_{p+q=n} H^p_{naive}(E)/(T_3) \otimes \X^{kil, q} \;,
\end{equation} where $(T_3)$ is the ideal in $H^\bullet_{naive}(E)$ which
is generated by the image $T_3(\X(N))$ of $T_3$.
\end{thm}
\begin{proof} According to Proposition~\ref{p:E2} and Corollary~\ref{c:E3}, 
the third sheet of the spectral sequence is given by $E_3^{p,q}\cong 
H_{naive}^p(E)\otimes \X^q(N)$. Note that $\X^q(N)= \Gamma_N(S^{q/2}(TN))$ 
is generated as an algebra by its degree 2 elements. Since the differential 
 $\ud_3:E_3^{p,q}\to E_3^{p+3,q-2}$ is a derivation, it is necessarily the 
unique derivation extending its restriction $T_3=\ud_3: \X(N)\to H_{naive}^3(E)
 $. Since $\X^\bullet(N)=\Gamma_N(S^{q/2}(TN))$ is a free graded commutative 
algebra, the cohomology $E_4^{\bullet,\bullet}=H^\bullet(E_3^{\bullet,\bullet},
 \ud_3)$ is given by
\begin{align*}
  E_4^{\bullet,q} &\cong H_{naive}^\bullet(E)/({\rm im}(T_3)) \otimes
    S^{q/2} (\ker T_3) \\
  &\cong  H^\bullet_{naive}(E)/(T_3) \otimes \X^{kil, q} \;.
\end{align*}
Now, it is sufficient to prove that all higher differentials $\ud_{r\geq 4}$ 
vanish. Since $\ud_4:E_4^{p,q}\to E_4^{p+4,q-3}$ is a derivation, it is 
completely determined by its restriction to the generators of $E_4^{\bullet,
\bullet}$ which lie in $E_4^{\bullet,0}$ and $E_4^{0,2}$. Thus, for obvious 
degree reasons, $\ud_4=0$ and similarly for all $\ud_{r>4}$. Therefore, the 
spectral sequence collapses at the fourth sheet\,: $E_\infty^{p,q}\cong E_4^{p,
 q}$.
\end{proof}
\begin{rem}\label{rem:algebra}
Theorem~\ref{thm:split} gives an isomorphism of vector spaces, not of
algebras in general. However,  by standard results on spectral sequences of
algebras, the  isomorphism~\eqref{eq:split} is an isomorphism of graded
algebras if the right hand side of~\eqref{eq:split} is free as a graded
commutative algebra.
\end{rem}
\begin{rem} Theorem~\ref{thm:split} implies that all the cohomological 
information of a Courant algebroid with split base is encoded in the 
transgression homomorphism $T_3$ together with the naive cohomology (and image 
of the anchor map). We like to think of $T_3$ as a family of closed 3-sections 
of $\ker \rho$ obtained by transgression from $\X(N)$. This idea is made more 
explicit in the case of generalized exact Courant algebroids in 
Section~\ref{s:gexsplit}. In that case, the transgression homomorphism is 
closely related to a generalization of the characteristic class of the Courant 
algebroids as defined by  {\v S}evera~\cite{SevLett}, that is the 
cohomology class of the structure 3-form (see Proposition~\ref{P:Severaclass} 
and Proposition~\ref{P:D=C}) parameterizing such Courant algebroids.
\end{rem}

\section{Generalized exact Courant algebroids with split base}\label{s:gexsplit}
In this section, we consider a generalization of exact Courant
algebroids. These Courant algebroids are parametrized by the cohomology class
of closed $3$-forms from which an explicit formula for the transgression
homomorphism $T_3$ can be given.

An \emph{exact}
Courant algebroid $E\to M$ is a Courant algebroid  such  that the following
sequence 
 $$ 0\to T^*M \xrightarrow{\rho^*} E\xrightarrow{\rho} TM\to 0 $$
is exact.

Assume that $D:=\im\rho\subset TM$ is a subbundle, \emph{i.e.}, $E$ is
regular. Then the anchor maps surjectively $E\stackrel{\rho}\to D$ and its
dual $\rho^T:T^*M\to E^*\cong E$ factors through an injective map
 $D^*\stackrel{\rho^*}\to E$ (again $E$ and $E^*$ are identified by the
pseudo-metric). 
\begin{vdef}\label{d:genexact} A regular Courant algebroid
  $(E,\rho,\ldots)$ is generalized exact if the following sequence of bundle
  morphisms over $M$
\begin{eqnarray}\label{eq:genexact} 0\to D^*\xrightarrow{\rho^*} E\xrightarrow{\rho} D\to 0 \end{eqnarray}
is exact.
\end{vdef}
Based on the above discussion, the only condition to check in order for the
sequence~\eqref{eq:genexact} to be exact is exactness in $E$.

\medskip

There is a simple geometric classification of exact Courant algebroids due to
{\v S}evera~\cite{SevLett} which extends easily to generalized exact Courant
algebroids as follows. Note that $D$ is Lie subalgebroid of $TM$. Given a Lie algebroid $D$, we write $(\Omega^\bullet(D),\ud_D)$ the complex of $D$-forms, $\Hald^\bullet(D)$ its cohomology and $Z^\bullet(D)=\ker(\ud_D)$ the closed $D$-forms~\cite{MacK87, CaWe99}.

Let $(E\to M,\rho,[.,.],\langle.,.\rangle)$ be a generalized exact Courant
algebroid  and assume given a splitting of the exact
sequence~\eqref{eq:genexact} as a pseudo-Euclidean vector bundle, that is, an
isotropic (with respect to the pseudo-metric) section $\sigma:D\to E$ of
 $\rho$. Thus the section
 $\sigma$ identifies $E$ with $D^*\oplus D$ endowed with its standard
pseudo-metric: $\langle \alpha\oplus X, \beta\oplus Y\rangle=
 \alpha(Y)+\beta(X)$. Since $\rho$ preserves the bracket, for any 
 $X,Y\in \Gamma(D)$, one has  $$[\sigma(X),\sigma(Y)]=\sigma([X,Y]_{TM}) \oplus
\tilde{C}_\sigma(X,Y)$$ where $\tilde{C}_\sigma(X,Y) \in \rho^*(D^*)$. Let
 $C_\sigma$ be the dual of $\tilde{C}_\sigma$, that is, for $X,Y,Z\in
\Gamma(D)$, we 
define $C_\sigma(X,Y,Z)=\langle \tilde{C}_\sigma(X,Y),Z\rangle$. It follows
from axiom~\eqref{ax:nSkew} and axiom~\eqref{ax:adInv} of a Courant
algebroid (see Definition~\ref{d:Courant}) that $C_\sigma$ is
skew-symmetric. Moreover, by axiom~\eqref{ax:Leibn} and
axiom~\eqref{ax:adInv}, $C_\sigma$ is $\smooth(M)$-linear. Thus $C_\sigma$
is indeed a $3$-form on the Lie algebroid $D$,\emph{ i.e.}, $C_\sigma\in
\Omega^3(D)$. Furthermore, the (specialized) Jacobi identity, \emph{i.e.},
axiom~\eqref{ax:Jacobi} implies that $C_\sigma$ is closed, that is,
 $\ud_D(C_\sigma)=0$.
\begin{prop}[The {\v S}evera characteristic class] \label{P:Severaclass} Let 
 $E$ be a generalized exact Courant algebroid.
\begin{enumerate}\item There is a splitting of the exact sequence~\eqref{eq:genexact} 
as a pseudo-Euclidean bundle; in particular there is an isotropic section 
 $\sigma: D\to E$ of $\rho$.
\item \label{Severaclass} If $\sigma':D\to E$ is another isotropic section, 
then $C_\sigma -C_{\sigma'}$ is an exact $3$-form, that is, 
 $C_\sigma -C_{\sigma'}\in \im \ud_D$.
\end{enumerate}
\end{prop}
\begin{proof}  The proof is the same as the one for exact Courant algebroids; 
for instance, see {\v S}evera Letter~\cite{SevLett} or~\cite[Section 3.8]{Royt99}.
\end{proof}
In particular the cohomology class $[C_\sigma]\in \Hald^3(D)$ is independent of 
 $\sigma$. We will simply denote it $[C]$ henceforth.  We call the class $[C]\in \Hald^3(D)$ the 
{\v S}evera class of $(E\to M,\rho,[.,.],\langle.,.\rangle)$.

\smallskip

Given a closed $3$-form $C \in \Omega^3(D)$, one can define a bracket on the 
pseudo-Euclidean vector bundle $D^*\oplus D$ given, for $X,Y\in \Gamma(D)$,
 $\alpha, \beta \in \Gamma(D^*)$, by the formula
\begin{equation}\label{eq:formbracket}
  [\alpha \oplus X, \beta \oplus Y] = \Lie_X\beta -\imath_Y\ud_D(\alpha)+C(X,Y,.) \oplus [X,Y]_{TM}.
\end{equation}
It is straightforward to check that this bracket makes the pseudo-Euclidean 
bundle $D^*\oplus D$ a Courant algebroid, where the anchor map is the 
projection $D^*\oplus D\to D$~\cite[Section 3.8]{Royt99}. Clearly its 
{\v S}evera class is $C$. Moreover two cohomologous closed 
 $3$-forms $C,C'\in \Omega^3(D)$ yield isomorphic Courant 
algebroids~\cite[Section 3.8]{Royt99}. Therefore 
\begin{prop}[Analog of the {\v S}evera classification]\label{P:classification} 
Let $D$ be a Lie 
subalgebroid of a smooth manifold $M$. The isomorphism classes of generalized 
exact Courant algebroids with fixed image $\im \rho =D$ are in one to one 
correspondence with $\Hald^3(D)$, the third cohomology group of the Lie 
algebroid $D$.

The correspondence assigns to a Courant algebroid $E$ its {\v S}evera 
characteristic class $[C]$.
\end{prop}

\begin{rem}\label{rem:Hgenexact} Generalized exact Courant algebroids are easy 
to describe in terms of the derived bracket construction. Indeed, since we can 
choose an isomorphism $E\cong D\oplus D^*$, the minimal symplectic realization 
of $E$ is isomorphic to $T^*[2]D[1]$. From the explicit 
formula~\eqref{eq:formbracket} for the Courant bracket, we found that the 
generating cubic Hamiltonian $H$ (encoding the Courant algebroid structure) is 
given, in our adapted coordinates, by 
\begin{equation} \label{eq:Hgenexact} 
  H=p_I\xi^I+\frac16C_{IJK}\xi^I\xi^J\xi^K 
\end{equation} where $C_{IJK}$ are 
the components of the {\v S}evera $3$-form $C$ induced by the splitting $\cong 
D\oplus D^*$.
\end{rem}

\medskip

Now assume that $E$ is a generalized exact Courant algebroid with split base 
 $D\cong TL\times N$. Then, there is an isomorphism $\Omega^\bullet(D)\cong 
\Omega^\bullet(L)\otimes_\R\smooth(N)$ and the de Rham differential $\ud_D$ of 
the Lie algebroid $D$ is identified with $\ud_L \otimes 1$, the de Rham 
differential of the smooth manifold $L$, see Remark~\ref{rem:deRham}. \notneeded{In other 
words, the de Rham complex of the Lie algebroid $D$ is the de Rham complex of 
the manifold $L$ tensored by smooth functions on $N$.} In particular, $Z^3(D)
\cong Z^3(L)\otimes_\R\smooth(N) $. Furthermore, since $\ker \rho \cong 
D^*\cong T^* L\times N$, the naive complex $\big(\Gamma_M(\Lambda^\bullet \ker 
\rho),\ud\big)$ is isomorphic to the complex $\big(\Omega^\bullet(L)\otimes 
\smooth(N), \ud_L\otimes 1\big)$ of de Rham forms on $L$ tensored by smooth 
functions on $N$. Therefore $H_{naive}^\bullet(E)\cong H^\bullet_{DR}(L)\otimes 
\smooth(N)$, where $H^\bullet_{DR}(L)$ is the de Rham cohomology of the
manifold $L$.

\begin{Exam} \label{e:genexactsplitbase}
Let $L$ and $N$ be two smooth manifolds and define $M:=L\times N$, 
$D:=TL\times N$ which is a subbundle of $TM$. Pick $\omega\in Z^3(L)$ any 
closed $3$-form on $L$ and $f\in \smooth(N)$ be any function on $N$. Then 
$C:=\omega \otimes f$ is a closed $3$-form on $D$. The $3$-form $C$ induces a 
generalized exact Courant algebroid with split base structure on $E:= D\oplus 
D^*$ where the Courant bracket is given by formula~\eqref{eq:formbracket}, the 
pseudo-metric is the standard pairing between $D$ and $D^*$ and the anchor map 
the projection $E\to D$ on the first summand. By 
Proposition~\ref{P:classification}, any generalized exact Courant algebroid 
with split base is isomorphic to such a Courant algebroid.
\end{Exam}

\bigskip

 For any vector field $X\in \X(N)$ on $N$, there is the map 
\begin{equation} \label{eq:T3gen} \Omega_{naive}^\bullet(E)\cong
  \Omega^\bullet(L)\otimes \smooth(N) \xrightarrow{\; 1\otimes X\; }
  \Omega^\bullet(L)\otimes \smooth(N)\cong
  \Omega_{naive}^\bullet(E)
\end{equation} 
defined, for $\omega \in \Omega^\bullet(L)$ and $f\in \smooth(N)$ by
 $(1\otimes X)(\omega\otimes f)=\omega \otimes X[f]$. 
Applying this map to the {\v S}evera $3$-form $C\in Z^3(L)\otimes
\smooth(N)\cong Z^3(D)$ yields the map  $$\X(N) \ni X\mapsto [(1 \otimes X)(C)]\in
H^3_{naive}(E)$$ which is well defined and depends only on the {\v S}evera
class of $E$ (and not on the particular choice of a $3$-form
representing it) by  Proposition~\ref{P:Severaclass}.
\begin{prop}\label{P:D=C}
 Let $E$ be a generalized exact Courant algebroid with {\v S}evera
 class $[C]\in \Hald^3(D)\cong H^3_{DR}(L)\otimes \smooth(N)$. The
 transgression homomorphism $T_3:\X(N)\to H^3_{naive}(E)$ is the map given,
 for any vector field $X\in \X(N)$, by 
 $$T_3(X) = [(1 \otimes X)(C)] \in H^3_{naive}(E).$$
\end{prop}
\begin{proof} Fix a {\v S}evera $3$-form $C$ representing the {\v S}evera
  characteristic class. The transgression homomorphism is induced by the
  differential $\ud_3:E_3^{\bullet,\bullet}\to E_3^{\bullet+3,\bullet-2}$.
  Corollary~\ref{c:E3} and Proposition~\ref{p:E2} together with above remarks
  about generalized exact Courant algebroids give\,:
  $E_3^{p,q}\cong\Gamma(\Lambda^p\im\rho^*)\otimes\X^q(N).$  Therefore our
  adapted coordinates \ref{rem:lcoo} still apply and give on a local chart
  the map $\ud_3$ as
\begin{align*}
 \ud_3
  =& \frac16C_{IJK,L'}\xi^I\xi^J\xi^K\pfrac{}{p_{L'}}
\end{align*} \MG{Note that there are no further $\xi^A$, because of
  generalized exactness.}
By formula~\eqref{eq:Hgenexact}, the functions
 $\frac16C_{IJK}\xi^I\xi^J\xi^K$ are given by the components of the
{\v S}evera $3$-form $C$ (identified with a function on 
 $E[1]$ via $\langle.,.\rangle$ and pulled back to $\E$). Now the result
follows since $T_3$ is the restriction of $\ud_3$ to $E_3^{0,2}\cong \X(N)$.
\end{proof}

\medskip

 We denote $\Ann(C)$ the kernel of $T_3$, that is the vector fields $X$ on $N$ such that $[(1\otimes X)(C)]=0\in H^3_{naive}(E)$.  We also denote $\big((1\otimes \X(N))(C)\big)$ the ideal in $H^\bullet_{naive}(E)$ generated by the vector subspace $\im T_3 =\{(1\otimes X)(C) \, , \, X\in \X(N)\}$.
\begin{cor}\label{c:genexact} The Courant algebroid cohomology of a generalized exact Courant
 algebroid $E$ with split base $D\cong TL\times N$  is given by
\begin{align*} H^n_{std}(E)
  &\cong \bigoplus_{p+2q=n} H^p_{DR}(L)\otimes \smooth(N)/\big((1\otimes \X(N))(C)\big) \otimes_{\smooth(N)}
    S^{q}(\Ann(C))
\end{align*}
where $[C]$ is the  {\v S}evera class of $E$.
\end{cor}
\begin{proof} This is an immediate consequence of Theorem~\ref{thm:split}, 
Proposition~\ref{P:D=C} and the isomorphism $H^\bullet_{naive}(E)\cong 
H^\bullet_{DR}(L)\otimes \smooth(N)$.
\end{proof}

\begin{rem}\label{e:D3triv}
 Let $E$ be a generalized exact Courant algebroid with split base $D\cong 
 TL\otimes N$. Assume that the {\v S}evera class of $E$ can be represented by a 
  $3$-form $C\in \Omega^3(L)\otimes \gsmooth(N)$ which is constant as a function 
 of $N$, \emph{i.e.}, $C\in \Omega^3(L)\otimes \R \subset 
 \Omega^3(L)\otimes \gsmooth(N)$. Then, by Proposition~\ref{P:D=C} $T_3=0$ and 
 thus, by Corollary~\ref{c:genexact}, the cohomology of $E$ is 
  $$H^n_{std}(E) \cong \bigoplus_{p+q=n} H^p_{DR}(L)\otimes \X^{q}(N) \,.$$ 
 An example of such a 
 Courant algebroid is obtained as in Example~\ref{e:genexactsplitbase} by 
 taking $C=\omega \otimes 1$ for the {\v S}evera $3$-form, where $\omega$ is 
 any closed $3$-form on $L$.
\end{rem}

\begin{Exam}  Let $G$ be a Lie group with a bi-invariant metric
 $\langle.,.\rangle$. Then $G$ has a canonical closed $3$-form which  is the
Cartan $3$-form $C=\langle [\theta^L,\theta^L],\theta^L\rangle$ where
 $\theta^L$ is the left-invariant Maurer--Cartan $1$-form\MG{ and bracket and 
metric are understood on the Lie algebra component}. Note that, by
ad-invariance of  $\langle.,.\rangle$, $C$ is bi-invariant (hence closed).
\MG{That's also what Roytenberg said.  Seems I'm the only one who doesn't know 
this triviality.}

Thus, by Example~\ref{e:genexactsplitbase}, there is a generalized exact
Courant algebroid structure on $G\times N$ with {\v S}evera class $[C\otimes
  f]$ for any manifold $N$ and function $f\in\smooth(N)$. This example (for
 $N=\{*\}$) was suggested by  Alekseev (see also \cite[example
  5.5]{Royt02}).  Explicitly, the Courant algebroid is $E=(\g\oplus\g)
\times N\to G\times N$. The structure maps are given by\,:
\begin{align*}
  \langle X\oplus Y, X'\oplus Y'\rangle = \langle X,X'\rangle -\langle
  Y,Y'\rangle  \\
  \rho: E\to TG\boxplus TN: (X\oplus Y,g,n) \mapsto (X^l-Y^r)(g)\boxplus0
\intertext{where the superscript $^l$ (resp.\ $^r$) means that an element of
  the Lie algebra $\g\cong T_eG$ is extended as a left (right) invariant
  vector field on $G$. The bracket is given for $(g,n)\in G\times N$ and 
  $X,X',Y,Y'\in\g\subset\Gamma(\g\times G),$ by
} 
  [X\oplus X', Y\oplus Y']_{(g,n)} = ([X,X']\oplus  f(n)[Y,Y'])_{(g,n)}&
\end{align*}  Choosing the splitting $\sigma(Z,g,n):=(Z\oplus-\Ad_gZ,g,n)$, one 
finds that the Cartan $3$-form $C$ is indeed the {\v S}evera 
class of $E$.

Now assume $G$ is a compact simple Lie group.  Then, $C$ spans $H^3(G)$. If we 
take $N=\R$ and $f$ to be constant, then, by 
Remark~\ref{e:D3triv}, the cohomology of $E$ is $H^\bullet_{std}(E)\cong H^\bullet(G)
\otimes \X^\bullet(\R)$, thus two copies (in different degrees) of the 
cohomology of $G$. Now take $f\in \smooth(\R)$ to be $f(t)=t$. Then $\Ann(C)$ 
is trivial and $ H^\bullet_{std}(E)= H^\bullet(G)/(C)$. Note that $H^\bullet(G)$
is the exterior algebra $H^\bullet(G)\cong 
 \Lambda^\bullet(C,x_2,\dots,x_r)$, thus $H^\bullet_{std}(E)\cong 
\Lambda^\bullet(x_2,\dots,x_r)$ as an algebra (see Remark~\ref{rem:algebra}). 
\end{Exam}

\chapter{Matched pairs of Courant algebroids}
\nocite{GSi09}
\section{Matched pairs}
\begin{vdef}  Given an anchored vector bundle $E\xrightarrow{\rho}TM$ and a 
  vector bundle $V$ over a smooth manifold $M$, an \emph{$E$-connection} on $V$ is an
 $\R$-bilinear operator 
  $\nabla:\Gamma(E)\otimes\Gamma(V)\to\Gamma(V)$ fulfilling:
\begin{gather}
  \conn_{f\psi}v = f\conn_\psi v , \label{dog} \\
  \conn_\psi(fv) = \big(\rho(\psi)[f]\big) v + f \conn_\psi v \label{cat} 
\end{gather}
for all $f\in\smooth(M)$, $\psi\in\Gamma(E)$, and $v\in\Gamma(V)$.
\end{vdef}

\begin{vdef}[Mokri \cite{Mor97}] Two Lie algebroids $A$ and $A'$ form a \emph{matched pair} 
if there exists a flat $A$-connection $\nablaright$ on $A'$ and a flat $A'$-connection $\nablaleft$ on $A$ 
satisfying 
\begin{gather}
  \lconn_\alpha[b,c] = [\lconn_\alpha b,c] +[b, \lconn_\alpha c] 
    +\lconn_{\rconn_c\alpha}b -\lconn_{\rconn_b\alpha}c \label{crocodile} \\
  \rconn_a[\beta,\gamma] = [\rconn_a \beta,\gamma] +[\beta, \rconn_a \gamma] 
    +\rconn_{\lconn_\gamma a}\beta -\rconn_{\lconn_\beta a}\gamma \label{alligator}
\end{gather}
(for all $\alpha,\beta,\gamma\in\Gamma(A')$ and $a,b,c\in\Gamma(A)$) are specified. 
\end{vdef}

Let $(E,\ip{.}{.},\rho,\dia)$ be a Courant algebroid. Assume we are given two subbundles $E_1,E_2$ of $E$ such that $E=E_1\oplus E_2$ and $E_1^\perp=E_2$. 
Let $\pr_1$ (resp.\ $\pr_2$) denote the projection of $E$ onto $E_1$ (resp.\ $E_2$) and 
$\i_1$ (resp.\ $\i_2$) denote the inclusion of $E_1$ (resp.\ $E_2$) 
into $E$. 
Assume that
 $E_k$ is itself a Courant algebroid with anchor $\rho_k=\rho\circ\i_k$, inner product 
$\ip{a}{b}_k=\ip{\i_k a}{\i_k b}$, and Dorfman bracket 
$\dbi{a}{b}{k}=\pr_k(\dor{\i_k a}{\i_k b})$.
A natural question is how to recover the Courant algebroid
structure on $E$ from  $E_1$ and $E_2$.

\begin{prop}  The inner product and the anchor map of $E$ are uniquely 
determined by their restrictions to $E_1$ and $E_2$.  Indeed, for $a,b\in\Gamma(E_1)$ 
and $\alpha,\beta\in\Gamma(E_2)$, we have 
\begin{gather}  
\<a\oplus\alpha,b\oplus\beta\> = \<a,b\>_1 +\<\alpha,\beta\>_2 , \label{bull} \\
\rho(a\oplus\alpha) = \rho_1(a)+\rho_2(\alpha) , \label{bear}
\end{gather}
where $a\oplus\alpha$ is the shorthand for $\i_1(a)+\i_2(\alpha)$.
\end{prop}

\begin{prop} 
The bracket on $E$ induces an $E_1$-connection on $E_2$: 
\beq{bird} \rconn_{a} \beta=\pr_2\big(\dor{(\i_1 a)}{(\i_2 \beta)}\big) \eeq
and an $E_2$-connection on $E_1$: 
\beq{fish} \lconn_{\beta} a=\pr_1\big(\dor{(\i_2 \beta)}{(\i_1 a)}\big) .\eeq
These connections preserve the inner products on $E_1$ and $E_2$: 
\begin{gather}
\rho(\alpha)\ip{a}{b}_1= \ip{\lconn_{\alpha}a}{b}_1 +\ip{a}{\lconn_{\alpha}b}_1 \label{cow} \\
\rho(a)\ip{\alpha}{\beta}_2= \ip{\rconn_a \alpha}{\beta}_2 +\ip{\alpha}{\rconn_{a}\beta}_2 \label{calf}
\end{gather}

Moreover, we have 
\beq{parrot} (a\oplus0)\dia(0\oplus\beta) = -\lconn_{\beta} a \oplus \rconn_{a} \beta \eeq
for $a\in\Gamma(E_1)$ and $\beta\in\Gamma(E_2)$.
\end{prop}

\begin{proof} The first two equations follow from the Leibniz rule \eqref{ax:Leibn} of a Courant algebroid.  The next two equations follow from the ad-invariance \eqref{ax:adInv}.  The last equation uses in addition axiom \eqref{ax:nSkew}.
\end{proof}

\begin{prop} 
\begin{enumerate}
\item For all $a,b\in\Gamma(E_1)$, we have 
\beq{pig} \dor{(a\oplus0)}{(b\oplus0)} = (\dbi{a}{b}{1}) \oplus 
\big( \tfrac{1}{2} \D_2 \ip{a}{b}_1 + \oo(a,b) \big) ,\eeq 
where $\oo\:\wedge^2\Gamma(E_1)\to \Gamma(E_2)$ is defined by the relation 
\beq{monkey} \oo(a,b)=\tfrac12 \pr_2 (\dor{\i_1 a}{\i_1 b}-\dor{\i_1 b}{\i_1 a}) .\eeq 
In fact, $\oo$ is entirely determined by the connection 
$\nablaleft:\Gamma(E_2)\otimes\Gamma(E_1)\to\Gamma(E_1)$ through the relation 
\beq{donkey} \ip{\gamma}{\oo(a,b)}_2 = \tfrac12 (\ip{\lconn_{\gamma} a}{b}_1 
- \ip{a}{\lconn_{\gamma} b}_1) \eeq 
\item For all $\alpha,\beta\in\Gamma(E_2)$, we have 
\beq{snake} \dor{(0\oplus\alpha)}{(0\oplus\beta)} = 
\big( \tfrac{1}{2} \D_1 \ip{\alpha}{\beta}_2 + \w(\alpha,\beta)\big) 
\oplus (\dbi{\alpha}{\beta}{2}) ,\eeq
where $\w\:\wedge^2\Gamma(E_2)\to \Gamma(E_1)$ is defined by the relation 
\beq{elephant} \w(\alpha,\beta)=\tfrac12 \pr_1 (\dor{\i_2 \alpha}{\i_2 \beta} 
- \dor{\i_2 \beta}{\i_2 \alpha}) .\eeq 
In fact, $\w$ is entirely determined by the connection 
$\nablaright:\Gamma(E_1)\otimes\Gamma(E_2)\to\Gamma(E_2)$ through the relation 
\beq{shark} \ip{c}{\w(\alpha,\beta)}_1 = \tfrac12 (\ip{\rconn_{c} \alpha}{\beta}_2 
- \ip{\alpha}{\rconn_{c} \beta}_2) .\eeq 
\end{enumerate}
\end{prop}

\begin{proof}  From axiom \eqref{ax:nSkew} we conclude that the symmetric part of the bracket is given by
\begin{align*}
  (a\oplus0)\dia(a\oplus0) &= \tfrac12\D\<a,a\>_1 = \tfrac12\D_1\<a,a\>_1 \oplus \tfrac12\D_2\<a,a\>_2 \;.
\intertext{Moreover, using axiom \eqref{ax:adInv} we get}
 \<0\oplus\gamma,(a\oplus0)\dia(b\oplus0)\> &= \rho(a\oplus0)\<0\oplus\gamma,b\oplus0\>-\<(a\oplus0)\dia(0\oplus\gamma),(b\oplus0)\> \\ &= \<\lconn_\gamma a,b\>_1 \;.
\end{align*}  The formula for $\mho$ occurs under analog considerations for $0\oplus\alpha$ and $0\oplus\beta$.
\end{proof}

As a consequence we obtain the formula 
\begin{multline} 
\dor{(a\oplus\alpha)}{(b\oplus\beta)} 
= \big( \dbi{a}{b}{1} +\lconn_{\alpha} b-\lconn_{\beta} a 
+\w(\alpha,\beta)+\tfrac12\D_1\ip{\alpha}{\beta}_2 \big) \\
\oplus \big( \dbi{\alpha}{\beta}{2} +\rconn_{a} \beta-\rconn_{b} \alpha 
+\oo(a,b)+\tfrac12\D_2\ip{a}{b}_1 \big)  \label{fullBracket} \;.
\end{multline}
This shows that the Dorfman bracket on $\Gamma(E)$ can be recovered from 
the Courant algebroid structures on $E_1$ and $E_2$ together with the 
connections  $\nablaleft$ and $\nablaright$.

\begin{lemma} \label{l:connExact}
For any $f\in\smooth(M)$, $b\in\Gamma(E_1)$, and $\beta\in\Gamma(E_2)$ we have $\rconn_{\D_1 f} \beta =0$ and $\lconn_{\D_2 f} b =0$.
\end{lemma}

\begin{proof} 
We have $ \conn_{\D_1 f}\beta = \pr_2((\D_1 f\oplus0)\dia(0\oplus\beta)) = \pr_2(\D f\dia(0\oplus\beta)) = 0 $
\end{proof}

Set
\begin{align}\label{dolphin} \Rright(a,b)\alpha &:= \rconn_{a}\rconn_{b} \alpha -\rconn_{b}\rconn_{a} \alpha -\rconn_{\dbi{a}{b}{1}}\alpha \\
\label{wale} \Rleft(\alpha,\beta)a &:= \lconn_{\alpha}\lconn_{\beta}a -\lconn_{\beta}\lconn_{\alpha}a -\lconn_{\dbi{\alpha}{\beta}{2}}a \;.
\end{align}
\begin{lemma}
The operators $\Rright$ and $\Rleft$ are sections of $\wedge^2 E_1^*\otimes\mathfrak{o}(E_2)\iso\wedge^2 E_1^*\otimes\wedge^2 E_2^*$, where $\mathfrak{o}(E_2)$ is the bundle of skew-symmetric endomorphisms of $E_2$.
\end{lemma}
\begin{proof} The $\smooth(M)$-linearity of $\Rright(a,b)\alpha$ in $\alpha$ follows from
 $$ \rho_1(a\dia_1 b) = [\rho_1(a),\rho_1(b)] \;.$$  $\Rright$ is also $\smooth(M)$-linear in $b$. In view of lemma \ref{l:connExact} it is also skew-symmetric under the exchange of $a$ and $b$.
\end{proof}

\begin{prop}\label{glass} Assume we are given two Courant algebroids $E_1$ and $E_2$ over the same manifold $M$ and two connections 
$\nablaright:\Gamma(E_1)\otimes\Gamma(E_2)\to\Gamma(E_2)$
 and $\nablaleft:\Gamma(E_2)\otimes\Gamma(E_1)\to\Gamma(E_1)$ preserving the fiberwise metrics and satisfying $\rconn_{\D_1 f} \beta =0$ and $\lconn_{\D_2 f} b =0$ for all $f\in\smooth(M)$, $b\in\Gamma(E_1)$, and $\beta\in\Gamma(E_2)$.
Then the inner product \eqref{bull}, the anchor map \eqref{bear}, and the bracket \eqref{fullBracket} on the direct sum $E=E_1\oplus E_2$ satisfy \eqref{ax:Leibn}, \eqref{ax:nSkew}, and \eqref{ax:adInv}. 
Moreover, the Jacobi identity \eqref{ax:Jacobi} is equivalent to the following group of properties:
\begin{gather}
  \begin{split}&\lconn_\alpha(a_1\dia_1 a_2)-(\lconn_\alpha a_1)\dia_1 a_2-a_1\dia_1(\lconn_\alpha a_2)
    -\lconn_{\rconn_{a_2}\alpha}a_1 +\lconn_{\rconn_{a_1}\alpha}a_2 \\
    &\qquad =-\w (\alpha,\oo(a_1,a_2)+\tfrac12\D_2\<a_1,a_2\>)
      -\tfrac12 \D_1\big\<\alpha,\oo(a_1,a_2)+\tfrac12 \D_2\<a_1,a_2\>\,\big\>
  \end{split}\label{derofBr1}\\
  \begin{split}&\rconn_{a}(\alpha_1\dia_2\alpha_2)-(\rconn_{a} \alpha_1)\dia_2\alpha_2-\alpha_1\dia_2(\rconn_{a} \alpha_2)
    -\rconn_{\lconn_{\alpha_2}a}\alpha_1 +\rconn_{\lconn_{\alpha_1}a}\alpha_2 \\
    &\qquad=-\oo(a,\w (\alpha_1,\alpha_2)+\tfrac12 \D_1\<\alpha_1,\alpha_2\>)
      -\tfrac12 \D_2\big\<a,\w (\alpha_1,\alpha_2)+\tfrac12 \D_1\<\alpha_1,\alpha_2\>\,\big\>
  \end{split}\label{derofBr2}\\
  \Rright+\Rleft = 0\label{curvCompat}\\
  \lconn_{\oo(a_1,a_2)} a_3 +\text{c.p.}=0 \label{new4X} \\
  \rconn_{\w (\alpha_1,\alpha_2)} \alpha_3 +\text{c.p.}=0  \label{new4Y}
\end{gather}
\end{prop}
\begin{proof} Axioms \eqref{ax:Leibn}, \eqref{ax:nSkew}, and \eqref{ax:adInv} are straightforward computations from the definition \eqref{fullBracket} using \eqref{donkey} and \eqref{shark}.  

To check the Jacobi identity, it is helpful to compute for triples of sections of $E_1$ and $E_2$ respectively, which gives 8 cases.  The equations \eqref{curvCompat}--\eqref{new4Y} occur when one multiplies the Jacobi identity with an arbitrary section of $E_1\oplus E_2$. This will be helpful in evaluating the terms of the Jacobi identity.
\end{proof}


\begin{vdef}\label{pair}
Two Courant algebroids $E_1$ and $E_2$ over the same manifold $M$ together with a pair of connections satisfying the properties listed in Proposition~\ref{glass} are said to form a \emph{matched pair}.
\end{vdef}

If $(E_1,E_2)$ is a matched pair of Courant algebroids, the induced Courant algebroid structure on the direct sum vector bundle $E_1\oplus E_2$ is called the \emph{matched sum Courant algebroid}.

\section{Examples}
\subsection{Complex manifolds}

Let $X$ be a complex manifold. Its tangent bundle $T_X$ is a holomorphic vector bundle. 
Consider the almost complex structure $j:T_X\to T_X$. 
Since $j^2=-\id$, the complexified tangent bundle $T_X\otimes\C$ decomposes as the direct sum of $T_X\zo$ and $T_X\oz$. Let $\pr\zo\:T_X\otimes\C\to T\zo$ and $\pr\oz\:T_X\otimes\C\to T\oz$ denote the canonical projections.
The complex vector bundles $T_X$ and $T_X\oz$ are canonically identified to one another by the bundle map $\tfrac12(\id-i j):T_X\to T_X\oz$. 

\begin{vdef} A \emph{complex Courant algebroid} is a complex vector bundle $E\to M$ endowed with a symmetric nondegenerate $\C$-bilinear form $\ip{.}{.}$ on the fibers of $E$ with values in $\C$, a $\C$-bilinear Dorfman bracket $\dia$ on the space of sections $\Gamma(E)$ and an anchor map $\rho:E\to TM\otimes\C$ satisfying relations 
\eqref{ax:Jacobi}, \eqref{ax:Leibn}, \eqref{ax:nSkew}, and \eqref{ax:adInv}
where $\D=\rho^*\circ\ud:\smooth(M;\C)\to\Gamma(E)$ and $\phi,\phi_1,\phi_2,\phi'\in\Gamma(E)$, $f\in\smooth(M;\C)$.
\end{vdef}

The complexified tangent bundle $T_X\otimes\C$ of a complex manifold $X$ is a smooth complex Lie algebroid. 
As a vector bundle, it is the direct sum of $T_X\zo$ and $T_X\oz$, which are both Lie subalgebroids. 

The Lie bracket of (complex) vector fields induces a $T_X\oz$-module structure on $T_X\zo$ 
\beq{horse} \rep_X Y=\pr\zo \lie{X}{Y} \qquad (X\in\Xx\oz, Y\in\Xx\zo) \eeq
and a $T_X\zo$-module structure on $T_X\oz$ 
\beq{deer} \rep_Y X=\pr\oz \lie{Y}{X} \qquad (X\in\Xx\oz, Y\in\Xx\zo) .\eeq
The flatness of these connections is a byproduct of the  integrability of $j$.
We use the same symbol $\rep$ for the induced connections on the dual spaces: 
\begin{gather} \rep_X \beta=\pr\zo (\mathcal{L}_X \beta) \qquad (X\in\Xx\oz, \beta\in\Omega\zo) ,\label{squirrel} \\ 
\rep_Y \alpha=\pr\oz (\mathcal{L}_Y \alpha) \qquad (Y\in\Xx\zo, \alpha\in\Omega\oz) .\label{frog}
\end{gather}

As in Example~\ref{ex:Sev},  associated to
 each $H^{3,0}\in\Omega^{3,0}(X)$
 such that $\partial H^{3,0}=0$, there is a complex twisted
 Courant algebroid structure on $C\oz_X=T\oz_X\oplus(T\oz_X)^*$.
  Moreover, if two $(3,0)$-forms $H^{3,0}$ and $H'^{3,0}$ are $\partial$-cohomologous, then the associated twisted
 Courant algebroid structures on $C\oz$ are isomorphic.
Similarly, associated to each $H^{0, 3}\in\Omega^{0, 3}(X)$ 
 such that $\bar{\partial} H^{0, 3}=0$, there is a complex twisted
 Courant algebroid structure on 
$C\zo_X=T_X\zo\oplus(T_X\zo)^*$.

\begin{prop} 
Let $H=H^{3,0}+H^{2, 1}+H^{1, 2}+H^{0, 3}\in \Omega^3_\C (X)$,
where $H^{i, j}\in \Omega^{i, j}(X)$, be a closed $3$-form.
Let $(C^{1,0}_X)_{H^{3,0}}$ be the complex
 Courant algebroid structure on $C\oz_X=T\oz_X\oplus(T\oz_X)^*$
twisted by $H^{3,0}$, and $(C^{0, 1}_X)_{H^{0, 3}}$ be the complex 
 Courant algebroid structure on
$C\zo_X=T_X\zo\oplus(T_X\zo)^*$ twisted by $H^{0, 3}$.
Then $(C^{1,0}_X)_{H^{3,0}}$ and $(C^{0, 1}_X)_{H^{0, 3}}$ 
form a matched pair of Courant algebroids, with
connections given by

\begin{align}
\rconn_{X\oplus\alpha}(Y\oplus \beta)
  =&\, \nabla^o_X Y \oplus \nabla^o_X \beta
    +H^{1, 2}(X,Y, \cdot) \label{plane} \\
  \lconn_{Y\oplus\beta}(X\oplus \alpha)
  =&\, \nabla^o_Y X \oplus \nabla^o_Y \alpha
    +H^{2, 1}(Y,X, \cdot) \label{car} 
\end{align} 
$\forall X\in \Xx^{1,0}$, $\alpha\in\Omega^{1,0}$, $Y\in\Xx^{0,1}$.

The resulting matched sum Courant algebroid is isomorphic to
the standard complex Courant algebroid $(T_X\oplus T_X^*)\otimes \C$
twisted by $H$.
\end{prop}
\begin{proof}  The Courant algebroid $(C\oz_X)_{H^{3,0}}$ is the generalized twisted complex standard Courant algebroid of the complex Lie algebroid $T\oz_X$.  Therefore the 3-form $H^{3,0}$ must be closed under $\partial$.  That is true, because the $(4,0)$-component of $\ud H$ vanishes.  The analog considerations are true for $(C\zo_X)_{H^{0,3}}$.

It is clear that $(T_X\oplus T^*_X)\otimes\C$ can be twisted by $H$.  Straightforward computations show that this induces the twisted standard brackets on $C\oz$ and $C\zo$.  Also the connections are induced by this Courant bracket.
\end{proof}

\subsection{Holomorphic Courant algebroids}
Let $X$ be a complex manifold.
We denote by $C\oz_X=T_X\oz\oplus(T_X\oz)^*$ and $C\zo_X=T_X\zo\oplus(T_X\zo)^*$ the generalized standard Courant algebroids.

\begin{vdef} 
A \emph{holomorphic Courant algebroid} consists of a holomorphic vector bundle $E$ over a complex manifold $X$, with sheaf of holomorphic sections $\mathcal{E}$, endowed with a $\C$-valued non-degenerate inner product $\ip{.}{.}$ on its fibers inducing a homomorphism of sheaves of $\holom_X$-modules $\mathcal{E}\otimes_{\holom_X}\mathcal{E}\xto{\ip{.}{.}}\holom_X$, a bundle map $E\to T_X$ inducing a homomorphism of sheaves of $\holom_X$-modules $\mathcal{E}\xto{\rho}\hX_X$, where $\hX_X$ denotes the sheaf of holomorphic sections of $T_X$, 
and a homomorphism of sheaves of $\C$-modules $\mathcal{E}\otimes_{\C}\mathcal{E}\xto{\dia}\mathcal{E}$ satisfying relations 
\eqref{ax:Jacobi}, \eqref{ax:Leibn}, \eqref{ax:nSkew}, and \eqref{ax:adInv},
where $\D=\rho^*\circ\partial:\holom_X\to\mathcal{E}$ and $\phi,\phi_1,\phi_2,\phi'\in\mathcal{E}$, $f\in\holom_X$.
\end{vdef} 
 
\begin{lemma}[see also Thm.~2.6.26 in \cite{Huy05}]\label{wind}
Let $E$ be a complex vector bundle over a complex manifold $X$, and let $\mathcal{E}$ be the sheaf of $\holom_X$-modules of sections of $E\to X$ such that each $x\in X$ has an open neighborhood $U\in X$ with $\Gamma(U;E)=\smooth(U,\C)\cdot\mathcal{E}(U)$. Then the following assertions are equivalent. 
\begin{enumerate}
\item The vector bundle $E$ is holomorphic with sheaf of holomorphic sections $\mathcal{E}$. 
\item There exists a (unique) flat $T_X\zo$-connection $\nabla$ on $E\to X$ such that 
\[ \mathcal{E}(U)=\{\sigma\in\Gamma(U;E) \text{ s.t. } \conn_Y\sigma=0,\;\forall Y\in\Xx\zo(X) \} .\]
\end{enumerate} 
\end{lemma}

\begin{lemma}\label{dream}
Let $E\to X$ be a holomorphic vector bundle with $\mathcal{E}$ being the sheaf of holomorphic sections, and $\nabla$ the corresponding flat $T_X\zo$-connection on $E$. Let $\ip{}{}$ be a smoothly varying $\C$-valued symmetric nondegenerate $\C$-bilinear form on the fibers of $E$. The following assertions are equivalent: 
\begin{enumerate}
\item The inner product $\ip{.}{.}$ induces a homomorphism of sheaves of $\holom_X$-modules $\mathcal{E}\otimes_{\holom_X}\mathcal{E}\to\holom_X$. 
\item For all $\phi,\psi\in\Gamma(E)$ and $Y\in\Xx\zo(X)$, we have 
\[ Y\ip{\phi}{\psi}=\ip{\conn_Y\phi}{\psi}+\ip{\phi}{\conn_Y\psi} .\] 
\end{enumerate}
\end{lemma}

\begin{lemma}\label{child}
Let $E\to X$ be a holomorphic vector bundle with $\mathcal{E}$ being the sheaf of holomorphic sections, and $\nabla$ the corresponding flat $T_X\zo$-connection on $E$.  Let $\rho$ be a homomorphism of (complex) vector bundles from $E$ to $T_X$. 
The following assertions are equivalent:
\begin{enumerate}
\item The homomorphism $\rho$ induces a homomorphism of sheaves of $\holom_X$-modules $\mathcal{E}\to\hX_X$. 
\item The homomorphism $\rho\oz:E\to T_X\oz$ obtained from $\rho$ by identifying $T_X$ to $T_X\oz$ satisfies the relation \begin{equation}\label{zzzb} \rho\oz(\conn_Y \phi)=\pr\oz\lie{Y}{\rho\oz\phi} , \quad \forall Y\in\Gamma(T_X\zo),\;\phi\in\Gamma(E\oz) .\end{equation}
\end{enumerate}
\end{lemma}

\begin{lemma} Let $(E,\ip{.}{.},\rho,\dia)$ be a holomorphic Courant algebroid over a complex manifold $X$.  Denote the sheaf of holomorphic sections of the underlying holomorphic vector bundle by $\mathcal{E}$ and the corresponding flat $T_X\zo$-connection on $E$ by $\nabla$. 

Then, there exists a unique complex Courant algebroid structure on $E\to X$ with inner product $\ip{.}{.}$ and anchor map 
$\rho\oz=\tfrac12(\id-i j)\circ\rho:E\to T_X\oz\subset T_X\otimes\C$, the restriction of whose Dorfman bracket to $\mathcal{E}$ coincides with $\dia$.
This complex Courant algebroid is denoted by $E^{1,0}$.
\end{lemma}
\begin{proof}  To prove uniqueness, assume that there are two complex Courant algebroid structures on the smooth vector bundle $E\to X$ whose anchor maps and inner products coincide.  Moreover the Dorfman brackets coincide on $\E$ the subsheaf of holomorphic sections.  Then the Dorfman brackets coincide on all of $\Gamma(E)$, because the holomorphic sections over a coordinate neighborhood of $X$ are dense in the set of smooth sections.

To argue for existence, note that we want the Leibniz rules
\begin{align*}
  \dor{\phi}{(g\cdot\psi)} &= \rho\oz(\phi)[g]\cdot\psi + g\cdot(\dor{\phi}{\psi})  \\
  \dor{(f\cdot\phi)}{\psi} &= -\rho\oz(\psi)[f]\cdot\phi +f\cdot(\dor{\phi}{\psi}) +\<\phi,\psi\>\cdot\rho^{(1,0)*}\ud f 
\end{align*}  for $\phi,\psi\in\Gamma(E)$ and $f,g\in\smooth(X)$.  Therefore we define the Dorfman bracket of the smooth sections $f\phi$, $g\psi$, for $\phi,\psi\in\E$ as
\begin{align*}
  \dor{(f\cdot\phi)}{(g\cdot\psi)} =& f\rho\oz(\phi)[g]\cdot\psi + fg\cdot(\dor{\phi}{\psi})  \\
  &-g\rho\oz(\psi)[f]\cdot\phi +\<\phi,\psi\>g\cdot\rho^{(1,0)*}\ud f  \;.
\end{align*}
First we should argue that this extension is consistent with the bracket defined for holomorphic sections.  Assume therefore that $f\phi$ is again holomorphic, where $f\in\smooth(M)$ and $\phi\in\E$.  But this implies that $f$ is holomorphic on the open set where $\phi$ is not zero.  Therefore on this open set the extension says
\begin{align*}
  f\rho\oz(\phi)[g]\cdot\psi &+ fg\cdot(\dor{\phi}{\psi}) 
  -g\rho\oz(\psi)[f]\cdot\phi +\<\phi,\psi\>g\cdot\rho^{(1,0)*}\ud f  \\
  &= f\rho(\phi)[g]\cdot\psi + fg\cdot(\dor{\phi}{\psi}) 
  -g\rho(\psi)[f]\cdot\phi +\<\phi,\psi\>g\cdot\rho^{*}\partial f 
\intertext{where we have used that the two anchor maps coincide for holomorphic sections and $\rho^{(1,0)*}\ud f=\rho^*\partial f$.  Therefore the term is just}
  \dor{(f\cdot\phi)}{(g\cdot\psi)} \;.
\end{align*} Since the first derivatives of $f$ and $g$ are continuous the above relation also holds for the closure of the open domain where both $\phi$ and $\psi$ do not vanish.  But on the complement, which is open, the term is just 0, because at least one of the terms $\phi$ or $\psi$ is also 0.

Then follows a straightforward but tedious computation that shows that this bracket fulfills the four axioms \eqref{ax:Jacobi}--\eqref{ax:adInv}.  As an example we show the proof of the axiom \eqref{ax:nSkew}:
\begin{align*}
  \dor{(f\cdot\phi)}{(f\cdot\phi)} &= f^2\cdot(\dor{\phi}{\phi}) +\<\phi,\phi\>f\cdot\rho^{(1,0)*}\ud f \\
  &= \tfrac12\rho^*\partial\<\phi,\phi\> +\<\phi,\phi\>f\cdot\rho^{(1,0)*}\ud f \\
  &= \tfrac12\rho^{(1,0)*}\ud\<f\phi,f\phi\>
\end{align*}  where in the last relation we used again the fact that $\rho^*\partial=\rho^{(1,0)*}\ud$ when applied to holomorphic functions.
\end{proof}

Define a flat $C_X\zo$-connection $\nablaright$ on $E$ by 
\beq{squalor} \rconn_{Y\oplus\eta} e=\conn_Y e \eeq
and a flat $E$-connection $\nablaleft$ on $C_X\zo$ by 
\beq{filth} \lconn_{e} (Y\oplus\eta)=\nabla^o_{\rho\oz(e)}(Y\oplus\eta)=\pr\zo\big(\mathcal{L}_{\rho\oz(e)}(Y\oplus\eta)\big) 
=\pr\zo\big(\lie{\rho\oz(e)}{Y}\oplus\mathcal{L}_{\rho\oz(e)}\eta\big) ,\eeq
where $\pr\zo$ denotes the projection of $(T_X\oplus T_X^*)\otimes\C$ onto $T_X\zo\oplus(T_X\zo)^*$.

We can write the identity \eqref{zzzb} as
\begin{equation}\label{zzzbb} \rho(\rconn_Y\phi)-\lconn_\phi Y = [Y,\rho(\phi)] \;.
\end{equation} \MG{Mathieu please check again, because that is what you erased from my proposition.}

Thus we have  the following

\begin{prop}  Let $(E,\ip{.}{.},\rho,\dia)$ be a holomorphic Courant algebroid over a complex manifold $X$.  Denote the sheaf of holomorphic sections of the underlying holomorphic vector bundle by $\mathcal{E}$ and the corresponding flat $T_X\zo$-connection on $E$ by $\nabla$. 

The two complex Courant algebroids $C_X\zo$ and $E\oz$, together with the 
flat connections $\nablaleft$ and $\nablaright$ given by \eqref{squalor} and \eqref{filth},
  constitute a matched pair of complex Courant algebroids, which we call the companion matched pair of the holomorphic Courant algebroid $E$. 
\end{prop}
\begin{proof}  We check that $C_X\zo\oplus E^{1,0}$ is a Courant algebroid. Verifying \eqref{ax:Leibn}--\eqref{ax:adInv} is straightforward.

It remains to check that the Jacobi identity holds.  By Proposition \ref{glass}, it suffices to check that the 5 properties \eqref{derofBr1}--\eqref{new4Y} hold.   All equations are trivially satisfied when  $a_i\in\E$ and $\alpha_i\in\bar{\Theta}\oplus\bar\Omega_X\subset \Gamma(C\zo_X)$.  So we are left to check that by multiplying the $a_i$ (or the $\alpha_i$ respectively) with smooth functions, we get the same additional terms on each side of the equations.  This will be demonstrated for \eqref{derofBr1}.  Set
\begin{align*}
  LHS(a_2) &:= \lconn_\alpha(a_1\dia a_2) - (\lconn_\alpha a_1)\dia a_2 -a_1\dia(\lconn_\alpha a_2) -\lconn_{\rconn_{a_2}\alpha} a_1 +\lconn_{\rconn_{a_1}\alpha} a_2 \;. \\
 \intertext{Then}
  LHS(f\cdot a_2) &= f\cdot LHS(a_2) +\left(\left[\rho(\alpha),\rho(a_1)\right]+\rho(\rconn_{a_1}\alpha) 
    -\rho(\lconn_\alpha a_1)\right)[f]\cdot a_2
\intertext{and the second term on the right hand side vanishes due to \eqref{zzzbb}.  Thus the left hand side of \eqref{derofBr1} is $\smooth$-linear in $a_2$.  Similarly, setting}
  RHS(a_2) &:= -\w(\alpha,\oo(a_1,a_2)+\tfrac12\ud\<a_1,a_2\>) -\tfrac12\D\<\alpha, \oo(a_1,a_2)+\tfrac12 \ud\<a_1,a_2\>\,\> \;,  \\
 \intertext{we get}
  RHS(f\cdot a_2) &= f\cdot RHS(a_2) \;.
\end{align*} 
Thus the right hand side of \eqref{derofBr1} is also $\smooth$-linear in $a_2$.  Analog considerations result in coinciding terms for both sides of \eqref{derofBr1} when multiplying $a_1$ or $\alpha$ with a smooth function.
\end{proof}

In fact, the converse is also true.

\begin{prop} 
Let $X$ be  a complex manifold. Assume that
$(C\zo_X,B)$ is a   matched pair of complex Courant algebroids 
 such that the anchor of $B$ takes values in $T_X\oz$,
both connections $\nablaleft$ and $\nablaright$ are flat with
the $B$-connection $\nablaleft$ on $C\zo_X$ 
being given by $\lconn_{e} (Y\oplus\eta)=\nabla^o_{\rho(e)}(Y\oplus\eta)$.
Then there is a unique holomorphic Courant algebroid   $E$ such that
$B=E^{1, 0}$.
\end{prop}
\begin{proof} The flat $C^{0,1}$-connection induces a flat $T^{0,1}$-connection on $B$.  Hence there is a holomorphic vector bundle $E\to X$ such that $E\oz=B$, according to Lemma \ref{wind}.  Since the connection $\nablaleft$ preserves the inner product on $B$ and is compatible with the anchor map \eqref{zzzb}, $E$ inherits a holomorphic inner product and a holomorphic anchor map.  It remains to check that the induced Dorfman bracket is holomorphic as well.  If $a_1,a_2\in\E$ are two holomorphic sections of $E$ and $\alpha\in\overline{\Theta}_X\oplus\overline{\Omega}_X$ is an anti-holomorphic section of $C\oz_X$, then all terms of equation \eqref{derofBr1} except the first one vanish.  But then the first one $\lconn_\alpha(a_1\dia a_2)$ also has to vanish. This shows that the Dorfman bracket of two holomorphic sections is itself holomorphic.
\end{proof}

\subsection{Flat regular Courant algebroid}

A Courant algebroid $E$ is said to be \textbf{regular} iff $F:=\rho(E)$ has constant rank, in which case $\rho(E)$ is an integrable distribution on the base manifold $M$ and $\ker\rho/(\ker\rho)^\perp$ is a bundle of quadratic Lie algebras over $M$. It was proved by Chen et al.\ \cite{CSX09}
that the vector bundle underlying a regular Courant algebroid $E$ is isomorphic to
$F^*\oplus \GG\oplus F$, where $F$ is the integrable subbundle $\rho(E)$ of $TM$ 
and $\GG$  the bundle $\ker\rho/(\ker\rho)^\perp$ of quadratic Lie algebras over $M$.
Thus we can confine ourselves to 
those Courant algebroid structures on $\fsgf$ where the anchor map is
\begin{equation} \label{Eqt:Standardrho} \rho(\xi_1+\br_1+x_1)=x_1 ,\end{equation}
the pseudo-metric is
\begin{equation} \label{Eqt:Standardip}
\ip{\xi_1+\br_1+\fx_1}{\xi_2+\br_2+\fx_2} = \duality{\xi_1}{\fx_2}+\duality{\xi_2}{\fx_1}+\ipG{\br_1}{\br_2} \;,\end{equation}
and the Dorfman bracket satisfies
\begin{equation} \label{Eqt:prGdb} \pr_{\GG} (\dor{\br_1}{\br_2})=\lbG{\br_1}{\br_2} \;.\end{equation}
Here $\xi_1,\xi_2\in F^*$, $\br_1,\br_2\in\GG$, and $\fx_1,\fx_2\in F$.
We call them \textbf{standard} Courant algebroid structures on $\fsgf$.

\begin{thm}{\cite{CSX09}}
\label{Thm:StandardCourantStructure}
A Courant algebroid structure on $\fsgf$,
with pseudo-metric \eqref{Eqt:Standardip} and anchor map
\eqref{Eqt:Standardrho}, and satisfying \eqref{Eqt:prGdb}, is
completely determined by the following data:
\begin{itemize}
\item an $F$-connection $\nablaDer$ on $\GG$,
\item a bundle map $\mathcal{R}:~\wedge^2 F\,\rightarrow\, \GG$,
\item and a $3$-form $H\in \Gamma(\wedge^3 F^*)$
\end{itemize}
satisfying the compatibility conditions
\begin{align}\label{ipGinvariant} & \L_x\ipG{\br}{\bs}=\ipG{\ConnectionDer_{x}{\br}}{\bs}+\ipG{\br}{\ConnectionDer_{x}{\bs}},\\
\label{lbGinvariant} & \ConnectionDer_{x}\lbG{\br}{\bs}=
\lbG{\ConnectionDer_x\br}{\bs}+\lbG{\br}{\ConnectionDer_x\bs},\\
\label{dCurvautre0}
& \big(\ConnectionDer_x\mathcal{R}(y,z)-\mathcal{R}([x,y],z)\big) +c.p.=0,\\
\label{CurvatureConnectionDer} &\ConnectionDer_x\ConnectionDer_y
\br-\ConnectionDer_y\ConnectionDer_x \br
-\ConnectionDer_{[x,y]}\br=\lbG{\mathcal{R}(x,y)}{\br}, \\
\label{dH}&\ud_F H=\ipG{\mathcal{R}}{\mathcal{R}}
\end{align}
for all $x,y,z\in \Gamma(F)$ and $\br,\bs\in\Gamma(\GG)$.
Here $\ipG{\mathcal{R}}{\mathcal{R}}$ denotes the $4$-form on $F$ given by
\[ \ipG{\mathcal{R}}{\mathcal{R}}(x_1,x_2,x_3,x_4) = \tfrac{1}{4} \sum_{\sigma\in S_4} \sgn(\sigma)
\ipG{\mathcal{R}(x_{\sigma(1)},x_{\sigma(2)})}{\mathcal{R}(x_{\sigma(3)},x_{\sigma(4)})}, \]
where $x_1,x_2,x_3,x_4\in F$.
\end{thm}

The Dorfman bracket on  $E=\fsgf$ is then given by
\begin{align}
\label{FSGFx1x2}
&\dor{\fx_1}{\fx_2}=H(\fx_1,\fx_2,\_)  + \mathcal{R}
(\fx_1,\fx_2)+[\fx_1,\fx_2], \\\label{FSGFg1g2}
&\dor{\br_1}{\br_2}=\mathcal{P}
(\br_1,\br_2)+\lbG{\br_1}{\br_2},\\\label{FSGFxi1g2}
&\dor{\xi_1}{\br_2}=\dor{\br_1}{\xi_2}=\dor{\xi_1}{\xi_2}=0,\\\label{FSGFx1xi2}
&\dor{\fx_1}{\xi_2}=\L_{\fx_1}\xi_2, \\\label{FSGFxi1x2}
&\dor{\xi_1}{\fx_2}=-\L_{\fx_2}\xi_1+
\ud_F\duality{\xi_1}{\fx_2},\\\label{FSGFx1g2}
&\dor{\fx_1}{\br_2}=-\dor{\br_2}{\fx_1}=-2\mathcal{Q}(\fx_1,\br_2)+
\ConnectionDer_{\fx_1}\br_2,
\end{align}
for all $\xi_1,\xi_2\in\Gamma(F^*)$, $\br_1,\br_2\in\Gamma(\GG)$,
$\fx_1,\fx_2\in\Gamma(F)$.
Here $\ud_F:\smooth(M)\to\Gamma(F^*)$ denotes the leafwise de~Rham differential,
and the map $\mathcal{P}:\Gamma(\GG)\otimes\Gamma(\GG)\to\Gamma(F^*)$ and
$\mathcal{Q}: \Gamma(F)\otimes\Gamma(\GG)\to
\Gamma(F^*)$ are defined by the relation
\begin{equation} \label{Eqt:definationofPForm}
\duality{\mathcal{P}(\br_1,\br_2)}{y} = 2\ipG{\br_2}{\ConnectionDer_y\br_1} \end{equation}
and
\begin{equation} \label{Eqt:definationofQ}
\duality{\mathcal{Q}(x,\br)}{y} = \ipG{\br}{\mathcal{R}(x,y)} .\end{equation}

\begin{vdef}
A standard Courant algebroid  $E=\fsgf$ is said to be flat
if $\ipG{\mathcal{R}}{\mathcal{R}}$ vanishes.
\end{vdef}

\begin{prop}
Let $E=\fsgf$ be a flat
standard Courant algebroid. Then $(F\oplus F^*)_{H}$ and $\GG$
form a matched pair of Courant algebroids, 
and  $E$ is isomorphic to  their matched sum, where
$(F\oplus F^*)_{H}$ denotes the twisted Courant algebroid
on $F\oplus F^*$ by the $3$-form $H$.  Here
the $(F\oplus F^*)_{H}$ connection on $\GG$
is given by

\[\rconn_{x+\xi}\br=\ConnectionDer_x\br \;,  \label{Schnecke}\]

while the $\GG$-connection on $(F\oplus F^*)_{H}$ is is given by

\[ \lconn_r (X\oplus\xi) = 2\mathcal{Q}(X,r)\oplus0  \label{Wurm}
\]

\end{prop}
\begin{proof} Due to the flatness of the Courant algebroid the 3-form $H\in\Omega^3_M(F)$ is closed.  Therefore we can construct the twisted standard Courant algebroid $(F\oplus F^*)_H$.  We compare the components of the Dorfman bracket of the sum Courant algebroid $\g\oplus(F\oplus F^*)_H$ with the components of the bracket \eqref{FSGFx1x2}--\eqref{FSGFx1g2} and see that for the choice \eqref{Schnecke} and \eqref{Wurm} both coincide.
\end{proof}

\section{Matched pairs of Dirac structures}

\begin{vdef}  A Dirac structure $D$ in a Courant algebroid $(E,\<.,.\>,\dia,\rho)$ with split signature is a maximal isotropic and integrable subbundle.
\end{vdef}

\begin{prop} Given a matched pair $(E_1,E_2)$ of Courant algebroids and Dirac structures $D_1\subset E_1$ and $D_2\subset E_2$, the direct sum $D_1\oplus D_2$ is a Dirac structure in the Courant algebroid $E_1\oplus E_2$ iff $\lconn_\alpha a\in\Gamma(D_1)$ and $\rconn_a \alpha\in\Gamma(D_2)$ for all $\alpha\in\Gamma(D_2)$ and $a\in\Gamma(D_1)$.
\end{prop}

\begin{proof} Maximal isotropy is obvious.  Both conditions are
  necessary, because $(a\oplus0)\dia(0\oplus\alpha)= -\lconn_\alpha a \oplus \rconn_a\alpha$.  It 
  remains to check that also the $E_2$-component of the Dorfman bracket of two sections of $D_1$ is  in $D_2$ (and vice versa).  Indeed, we have  
\[ \<0\oplus\alpha,(a\oplus0)\dia(b\oplus0)\> = \<\conn_\alpha a,b\>_1 \]
The RHS vanishes due to isotropy of $D_1$.  Therefore, due to maximal isotropy of $D_2$, $\oo(\alpha,\beta)$ must be in $D_2$.
\end{proof}

\begin{cor} Let $D_1$ (resp.\ $D_2$) be a Dirac structure in a Courant algebroid $E_1$ (resp.\ $E_2$). If $(E_1,E_2)$ is a matched pair of Courant algebroids and $(D_1,D_2)$ a matched pair of Dirac structures, then $(D_1,D_2)$ is a matched pair of Lie algebroids. 
\end{cor}

\appendix

\newpage\phantomsection
\bibliography{bibliogr}
\bibliographystyle{alphasorteprint}

\singlespace


\end{document}